\documentclass[]{interact}

\usepackage{xcolor}
\usepackage{pdfcomment}
\usepackage[labelfont=bf]{caption}
\usepackage{amsmath}

\DeclareMathOperator*{\argmin}{arg\,min}

\usepackage{epstopdf}

\usepackage{natbib}
\bibpunct[, ]{(}{)}{;}{a}{}{,}
\renewcommand\bibfont{\fontsize{10}{12}\selectfont}

\theoremstyle{plain}
\newtheorem{theorem}{Theorem}[section]
\newtheorem{lemma}[theorem]{Lemma}
\newtheorem{corollary}[theorem]{Corollary}
\newtheorem{proposition}[theorem]{Proposition}

\theoremstyle{definition}

\theoremstyle{remark}
\newtheorem{remark}{Remark}

\newcommand\cH{{\cal H}}

\newcommand\e{{\bf e}}

\newcommand\ov{\overline}
\newcommand\oo{\mbox{o}}
\newcommand\OO{\mbox{O}}

\newcommand\tr{\mbox{tr}}

\def\bbr{{\mathbb R}}

\def\text#1{\hbox{#1}}

\def\endproof{\mbox{\ $\qed$}}

\def\m{{\bf m}}

\def\E{{\bf E}}

\def\P{{\bf P}}

\def\C{{\bf C}}
\def\D{{\bf D}}

\def\L{{\bf L}}

\def\U{{\bf U}}
\def\V{{\bf V}}
\def\n{{\bf n}}
\def\s{{\bf s}}
\def\x{{\bf x}}

\def\u{{\bf u}}
\def\g{{\bf g}}
\def\r{{\bf r}}
\def\a{{\bf a}}
\def\b{{\bf b}}

\def\q{{\bf q}}

\def\v{{\bf v}}
\def\t{{\bf t}}
\def\Chi{{\bf 1}}

\def\d{\mathrm{d}}
\def\build #1_#2{\mathrel{\mathop{\kern 0pt #1}\limits_{#2}}} 
\newcommand{\wh}{\widehat}
\newcommand{\wt}{\widetilde}
\newcommand{\zs}[1]{{\mathchoice{#1}{#1}{\lower.25ex\hbox{$\scriptstyle#1$}}
		{\lower0.25ex\hbox{$\scriptscriptstyle#1$}}}}

\begin{document}

\articletype{ARTICLE}

\title{Truncated sequential  guaranteed estimation for the Cox-Ingersoll-Ross models}

\author{
\name{Mohamed BEN ALAYA\textsuperscript{a}, Thi-Bao Trâm NGÔ\textsuperscript{b} \thanks{CONTACT Thi-Bao Trâm NGÔ. Email: thibaotram.ngo@univ-evry.fr} and Serguei PERGAMENCHTCHIKOV\textsuperscript{a}}
\affil{\textsuperscript{a}Laboratoire de Math\'ematiques Rapha\"el Salem, 
	UMR 6085 CNRS-Universit\'e de Rouen Normandie, France; \textsuperscript{b}Laboratoire de Mathématiques et Modélisation d'Évry, CNRS, UMR
	8071, Université Évry Paris-Saclay,  France}
}

\maketitle

\begin{abstract}
The drift sequential parameter estimation problems for the Cox-Ingersoll-Ross (CIR) processes under the limited duration of observation are studied. 
Truncated sequential estimation methods 
for both scalar and {two}-dimensional parameter cases are proposed.
In the non-asymptotic setting, for the proposed truncated estimators, the properties of guaranteed mean-square
estimation accuracy are established.  
In the asymptotic formulation, when the observation time tends to infinity, it is shown that the proposed sequential procedures are asymptotically optimal among all possible sequential and non-sequential estimates with an average estimation time less than the fixed observation duration.
It also turned out that asymptotically, without degrading the estimation quality, they significantly reduce the observation duration compared to classical non-sequential maximum likelihood estimations based on a fixed observation duration.

\end{abstract}


	\section{Introduction}	\label{sec:Intr}

\subsection{Motivations}\label{sec:Intr-11}

In this paper, based on the sequential analysis approach, we develop 
a new  truncated estimation method based on observations within the  time interval $[0,T]$  of the  Cox-Ingersoll-Ross (CIR) process  defined 
through the following stochastic differential equation

\begin{equation}\label{sec:Intr.CIR}
	\d X_\zs{t}=(a-bX_\zs{t})\d t+\sqrt{\sigma X_\zs{t}}\d W_\zs{t},\quad X_\zs{0}=x>0, \quad  0\le t\le T\,, 
\end{equation} 
where $a>0$, $b>0$ and $\sigma>0$ are fixed parameters and $(W_\zs{t})_\zs{t\ge 0}$ is a standard Brownian motion. 
Similarly to 
\cite{BenAlayaetal2025},
we consider the sequential estimation problem for the parameters $a$ and $b$ under the condition that the diffusion coefficient $\sigma$ is known.
{It should be noted that in this case
the process \eqref{sec:Intr.CIR} is  ergodic (see, for example, in \cite{BAK2010}, for the details)
and
has the ergodic density which defined as}
\begin{equation}
	\label{ergdenssityCIR}
	\q_\zs{a,b}(z)=
	\frac{\beta^{\alpha}}{\Gamma(\alpha)}\,z^{\alpha-1}\,e^{- \beta z}\,\Chi_\zs{\{z\ge 0\}}\,,
\end{equation}
where $\Gamma(z)=\int_{0}^{+\infty}t^{z-1}e^{-t}\d t$,
$\alpha=2 a/\sigma$ and  $\beta=2 b/\sigma$.

{The CIR model} is very popular in many important applications such as interest rate modeling
(\cite{CoxIngersollRoss1985,LamLap97}), stochastic volatility stock markets \cite{Heston1993, BerdjanePergamenshchikov2013, NguyenPergamenshchikov2017})
and, moreover, in \cite{PergamenchtchikovTartakovskySpivak2022}
 discrete versions of CIR processes are used
 in the epidemic analysis. To obtain reliable statistical inferences within this model, estimating the unknown parameters with guaranteed accuracy properties is necessary in the non-asymptotic setting. It should be noted that  for the model of type  \eqref{sec:Intr.CIR} the usual maximum likelihood estimators are  nonlinear functions of observations and it is not clear how to study such functions directly on the fixed time interval.
   For these reasons to overcome these dificulties
  in \cite{BenAlayaetal2025, NovikovShiryaevKordzakhia2024}, sequential guaranteed estimation methods were developed to estimate the parameters of the model  \eqref{sec:Intr.CIR} with fixed estimation accuracy. 
Unfortunately, the proposed sequential procedures do not control the observation duration, which restricts their applications in many practical applications since, in practice, the duration of observation is usually bounded. For example, for the portfolio optimisation problems for financial markets with unknown parameters
in \cite{BerdjanePergamenshchikov2015}, it is shown that to construct optimal and robust financial strategies, it is necessary to use guaranteed truncated sequential estimators that can provide a fixed known estimation accuracy over fixed time intervals. For statistical models  in discrete time truncated sequential procedures are developed in
\cite{KonevPergamenshchikov1990} and for 
 the stochastic differential equations  with the bounded diffusion coefficients
such procedures were proposed  in
\cite{KonevPergamenshchikov1992,GaltchoukPergamenshchikov2011,GaltchoukPergamenshchikov2015, GaltchoukPergamenshchikov2022}
for parametric and nonparametric  problems.
 Unfortunately, these results can not be applied  to  stochastic differential equations  with unbounded  diffusion coefficients 
as, for example,  \eqref{sec:Intr.CIR}.  To study non-asymptotic estimation methods for such models, one needs to develop  new analytic tools based on the special form of this process. The main goal of this paper is  to develop  truncated guaranteed estimation methods 
 for the coefficients $a$ and $b$, in the both scalar and {two}-dimensional parameter cases,
  on the basis of the observations $(X_\zs{t})_\zs{0\le t\le T}$ of the process \eqref{sec:Intr.CIR}, 
 where the observation duration $T>0$ is fixed in advance.

\subsection{Main contributions}
In this paper, for the first time, the truncated sequential guaranteed methods were developed for the models \eqref{sec:Intr.CIR} 
with unbounded diffusion coefficients. The proposed estimators for the parameters $a$ and $b$  have a guaranteed non-asymptotic mean square estimation accuracy, which is found in the explicit form.  Moreover, 
 through the asymptotic analysis methods developed in  \cite{BenAlayaetal2025}
  it is shown that the proposed truncated guaranteed procedures are optimal in the minimax sense
  for local  and general quadratic risks. For the local risks, the optimality properties are established in the class of all sequential and non sequential estimators with the  
  observation duration less than $T$ when $T\to \infty$.  
   It is important to emphasize here that it is also established that the proposed truncated procedures asymptotically, without deteriorating the estimation quality, significantly reduce the observation period compared to usual non-sequential estimations.
 Moreover, for the general quadratic risk, the optimality properties are established in the class of all sequential procedures with mean observation time 
not exceeding the mean observation time of the proposed truncated procedures. 
 It should be noted here that this class is sufficiently large since it includes all possible sequential procedures that can use more than $T$ observations duration; only the mean observation duration has to be less than the mean observation duration of the proposed truncated procedures. 
 This means that any sequential procedure having the same mean observation duration can not improve the accuracy properties with respect to the proposed one.

\subsection{Organisation of the paper}
The rest of the paper is organized as follows. In Section \ref{sec:SqPrs}, we analyze the scalar truncated sequential estimation methods for the model \eqref{sec:Intr.CIR}.  In Section \ref{sec:MultPr}, we develop the two-step sequential estimation method
for the parameter vector $\theta=(a,b)^{\top}$ in the model \eqref{sec:Intr.CIR}. In Section \ref{sec:AsOpt},
we find conditions on the parameters of the process \eqref{sec:Intr.CIR} which provide the optimality properties in minimax sense for the proposed sequential procedures.  Section \ref{sec:CoIn} presents the
concentration inequalities for the CIR process. Some important conclusions are given in Section \ref{sec:conclusion}.
{Appendix \ref{sec:Appendix} contains some useful properties of the CIR process \eqref{sec:Intr.CIR} and some auxiliary lemmas recalled from \cite{BenAlayaetal2025}.}

\medskip

\section{Scalar truncated sequential procedures} \label{sec:SqPrs}

\subsection{Guaranteed estimation}
\noindent
First, we consider the estimation problem for the parameter $b$ in the process \eqref{sec:Intr.CIR}  in the case, when $a$ is known, i.e. $\theta=b$. 
In this case  $\E_\zs{\theta}$ is the expectation with respect to the distribution $\P_\zs{\theta}$ of the process \eqref{sec:Intr.CIR} with 
a fixed parameter $a$ and 
$b=\theta$. In this case
the Maximum Likelihood Estimator (MLE) for $\theta$ (see, for example, in \cite{BAK2013}) is
the non-linear function of the observations
defined as
\begin{equation}\label{sec:SqPrs-1-MLE}
	\wh{\theta}_\zs{T}=\frac{aT-X_\zs{T}+x}{\int^{T}_\zs{0}\, X_\zs{s}\d s}	
	\,.
\end{equation}
\noindent 
To study the estimation problem 
in a non-asymptoitical setting in the paper \cite{BenAlayaetal2025} it is proposed  the sequential 
estimation
procedure $\delta_\zs{H}=(\tau_\zs{H},\wh{\theta}_\zs{\tau_\zs{H}})$
defined
as 
{\begin{equation}\label{eq:SMLE-bbEs-1}
	\tau_\zs{H}=\inf\left\{t>0\,:\,\int_\zs{0}^tX_\zs{s}\d s \ge H\right\}\textrm{ and }\wh{\theta}_\zs{\tau_\zs{H}}=\frac{a\tau_\zs{H}-X_{\tau_\zs{H}}+x}{H}\,,
\end{equation}}
\noindent where $H>0$ is a fixed threshold. To ensure that the duration estimate does not exceed the fixed period of time $T$
we define the truncated sequential procedure 
\begin{equation}\label{eq:TrunEst--Prxcb}
	\wt{\delta}_\zs{H,T}=(\wt{\tau}_\zs{H,T},\wt{\theta}_\zs{H,T})\,,
\end{equation}
\noindent in which  the alternative stopping time $\wt{\tau}_\zs{H,T}$ and the corresponding
sequential
estimator 
$\wt{\theta}_\zs{H,T}$ are defined as 
\begin{equation}\label{eq:TrunEst-b}
	\wt{\tau}_\zs{H,T}=\tau_\zs{H}\wedge T\quad\textrm{and}\quad
	\wt{\theta}_\zs{H,T}
	=\wh{\theta}_\zs{\tau_\zs{H}}
	\,
	\Chi_{\{\tau_\zs{H}\le T\}},
\end{equation}
\noindent where $x\wedge y=\min(x,y)$ and  the notation $\Chi_\zs{A}$ states for the indicator of the set $A$. 
Now,
for any compact  $\Theta\subset]0,+\infty[$ we denote
\begin{equation}
	\label{sec:Para-1}
	\a_\zs{*}
	=
	\frac{a}{\b_\zs{max}}
	\,,\quad
	\b_\zs{min}
	=\min_\zs{\theta\in \Theta} \theta
	\quad\mbox{and}\quad
	\b_\zs{max}
	=\max_\zs{\theta\in \Theta} \theta
	\,.
\end{equation}

\noindent Moreover, we need the following threshold
\begin{equation}
	\label{sec:LL-m-1}
	\L_\zs{m}
	=3^{2m-1}
	\Big(
	2 \x_\zs{2m}
	+
	\sigma^{m}
	\big(m(2m-1)\big)^{m}\,
	\x_\zs{m}
	\Big)\,.
\end{equation}
\noindent
where
the parameters $\x_\zs{q}$ are defined in \eqref{MomentCIR-2}.
\begin{theorem}\label{Th.thm:b-1}
	For any  $T\ge 1$,  $0<H<\a_\zs{*} T$ and  integer $m\ge 2$ the sequential procedure \eqref{eq:TrunEst-b} 
	possesses  the following non-asymptotic mean square estimation accuracy
	\begin{equation}
		\label{ineq:2}
		\sup_{\theta\in\Theta}\E_\zs{\theta}
		\big(	\wt{\theta}_\zs{H,T} -\theta\big)^2
		\le \frac{\sigma}{H}+\frac{T^m\U_\zs{m}}{(\a_\zs{*}T-H)^{2m}}
		:=
		\e_\zs{m}(H,T)		
		\,,
	\end{equation}
	where  $\U_\zs{m}=\L_\zs{m} \b^{2}_\zs{max}/\b^{2m}_\zs{min}$.
\end{theorem}
\begin{proof}
	\noindent  Indeed,
	followed by \cite[Theorem 1]{BenAlayaetal2025}, we have
	\begin{align}
		\E_\zs{\theta}\big(\wt{\theta}_\zs{H,T}-\theta\big)^2&\le \E_\zs{\theta}(	\wh{\theta}_\zs{\tau_\zs{H}}
		-\theta)^2\Chi_{\{\tau_\zs{H}\le T\}}+\theta^2
		\P_\theta(\tau_\zs{H}> T)\nonumber\\[2mm]
		&\le \frac{\sigma}{H}+\theta^2\P_\theta(\tau_\zs{H}> T).\label{ineq:varbeta}
	\end{align}
	\noindent From this, we note that for $0<H<\a_\zs{*} T$
	\begin{equation}
		\label{ineq:prob-tau}
		\P_\zs{\theta} (\tau_\zs{H}>T)=
		\P_\zs{\theta}\left(\int^{T}_\zs{0} X_\zs{s}\d s<H\right)
		\le 	
		\P_\zs{\theta}\left(|\D_\zs{T}|> \a_\zs{*} T -H\right)
		\,,
	\end{equation}
	\noindent
	where  $\D_\zs{T}
		=
		\int_\zs{0}^T
		\left(
		X_\zs{s}- a/\theta
		\right) \d s$.
	Then, from \eqref{ineq:prob-tau} and  the concentration inequality \eqref{sec:CoIn-D-T-2-1}
	for  
	$0<H< \a_\zs{*} T$, we get
	\begin{equation}
		\label{sec:PP-tau-Up}	
		\sup_\zs{\theta\in\Theta}\,
		\P_\zs{\theta} \big(\tau_\zs{H}>T\big)
		\le \frac{\sup_\zs{\theta\in\Theta}\,\E_\zs{\theta} \D_\zs{T}^{2m}}{(\a_\zs{*}T-H)^{2m}}\le
		\frac{T^m}{(\a_\zs{*}T-H)^{2m}}
		\frac{\L_\zs{m}}{\b^{2m}_\zs{min}}\,,
	\end{equation}
	\noindent where $\L_\zs{m}$ is defined in \eqref{sec:LL-m-1}.
	Therefore, using this in \eqref{ineq:varbeta} the estimation accuracy can be estimated as
	\begin{align*}
		\sup_{\theta\in\Theta}\E_\zs{\theta}(	\wt{\theta}_\zs{H,T}-\theta)^2&\le
		\frac{\sigma}{H}+\sup_{\theta\in\Theta}\theta^2
		\frac{ T^m}{(\a_\zs{*}T-H)^{2m}}
		\frac{\L_\zs{m}}{\b^{2m}_\zs{min}}\,.
	\end{align*}
	\noindent 
	This implies directly the bound \eqref{ineq:2}.
\end{proof}

\noindent Now one needs to choose an optimal value for the parameter $H$ 
to minimise the estimation accuracy \eqref{ineq:2}, i.e.
\begin{equation}
	\label{sec;Opt-H}
	H^{*}_\zs{T}=\argmin_\zs{0<H<\a_\zs{*}T} \e_\zs{m}(H,T)\,.
\end{equation}
\noindent In this case we define  the procedure
\begin{equation}
	\label{sec;Opt-H-Seq-Pr}
	\big(\tau^{*}_\zs{T},\theta^{*}_\zs{T} \big)\,,\quad
	\tau^{*}_\zs{T}=\wt{\tau}_\zs{H^{*}_\zs{T},T} 
	\quad\mbox{and}\quad
	\theta^{*}_\zs{T}=
	\wt{\theta}_\zs{H^{*}_\zs{T},T}
	\,.
\end{equation}

\begin{corollary}\label{cor:b}
	For  any integer $m>1$ the optimal truncated procedure  \eqref{sec;Opt-H-Seq-Pr} posses the following asymptotic 
	properties:.
	
	\begin{enumerate}
		
		\item 
		the optimal parameter
		\eqref{sec;Opt-H} is represented as
		\begin{equation}\label{H1}
			H^{*}_\zs{T}=\a_\zs{*}T-(2m\U_\zs{m}{\a_\zs{*}^2}/\sigma)^{\frac{1}{2m+1}}T^{\frac{2+m}{2m+1}}(1+\mathrm{o}(1))
			\quad\mbox{as}\quad T\to\infty \,;
		\end{equation} 
		\item
		\noindent 
		the corresponding optimal estimation accuracy 
		for any $m>1$
		has the following form 
		\begin{equation}\label{accuracy-H2}
			\sup_{\theta\in\Theta}\E_\zs{\theta}\big( \theta^{*}_\zs{T}-\theta\big)^2\le
			\e_\zs{m}(H^{*}_\zs{T},T)=
			\frac{\sigma}{\a_\zs{*}T}+\mathrm{O}\left(\frac{1}{T^{\frac{3m}{2m+1}}}\right)
			\quad\mbox{as}\quad T\to\infty \,.
		\end{equation}
	\end{enumerate}
\end{corollary}
\begin{proof}
	\noindent First note that to calculate  the parameter \eqref{sec;Opt-H} one needs to 
	study the  equation
	$$
	\frac{\partial }{\partial H}\,
	\e_\zs{m}(H,T)
	=0
	\,.
	$$ 
	\noindent 
	Using the form of the function $\e_\zs{m}(H,T)$
	defined
	in
	\eqref{ineq:2}, the root of this equation can be represented as
	\begin{equation}\label{eq:H1--1}
		H^{*}_\zs{T}=\a_\zs{*}T-(2m\U_\zs{m}/\sigma)^{\frac{1}{2m+1}}\,
		(H^{*}_\zs{T})^{\frac{2}{2m+1}}
		{T^{\frac{m}{2m+1}}}.
	\end{equation}
	\noindent Taking into account here that $H^{*}_\zs{T}<\a_\zs{*}T$ the parameter $H^{*}_\zs{T}$ 
	can be estimated from below as
	\begin{equation}\label{eq:H1-1}
		H^{*}_\zs{T}>\a_\zs{*}T-(2m\U_\zs{m}{\a_\zs{*}^2}/\sigma)^{\frac{1}{2m+1}}T^{\frac{2+m}{2m+1}}
		= \a_\zs{*}T
		\left(
		1-(2m\U_\zs{m}{\a_\zs{*}^{1-2m}}/\sigma)^{\frac{1}{2m+1}}T^{-\frac{m-1}{2m+1}}  
		\right) \,.
	\end{equation}
	\noindent
	Moreover,  using this bound  in \eqref{eq:H1--1} 
	the parameter $H^{*}_\zs{T}$  can be estimated from above as
	\begin{equation}\label{eq:H1-2}
		H^{*}_\zs{T}<\a_\zs{*}T-
		(2m\U_\zs{m}{\a_\zs{*}^2}/\sigma)^{\frac{1}{2m+1}}\,
		T^{\frac{2+m}{2m+1}}
		\left(1-(2m\U_\zs{m}{\a_\zs{*}^{1-2m}}/\sigma)^{\frac{1}{2m+1}}T^{-\frac{m-1}{2m+1}}\right)^{\frac{2}{2m+1}}\,
		\,.
	\end{equation}
	\noindent
	Taking into account that
	for any $m\ge 2$ the fraction
	${(2+m)}/{(2m+1)}<1$  
	and using the bounds
	\eqref{eq:H1-1} and \eqref{eq:H1-2} we can deduce that for sufficiently large $T$
	$$
	H^{*}_\zs{T}=\a_\zs{*}T-(2m\U_\zs{m}{\a_\zs{*}^2}/\sigma)^{\frac{1}{2m+1}}T^{\frac{2+m}{2m+1}}(1+\mathrm{o}(1)
	)\,.
	$$ 
	\noindent
	Therefore, 
	using this form in
	the bound
	\eqref{ineq:2}, we the representation \eqref{accuracy-H2}.
\end{proof}

\begin{remark} 
	\label{Re.sec:non-as--1-00}
	It should be noted (see, for example, in \cite{BenAlayaetal2025} that in this case the Fisher information  is represented  as
	\begin{equation}
		\label{sec:Fisher-Inf-1}
		I_\zs{a}(\theta)=\frac{a}{\sigma \theta}
		\quad\mbox{and}\quad
		I_\zs{a,*}=
		\min_\zs{\theta\in\Theta}\,I_\zs{a}(\theta)
		=I_\zs{a}(\b_\zs{max})=\frac{\a_\zs{*}}{\sigma}
		=\frac{a}{\sigma \b_\zs{max}}
		\,.
	\end{equation}
	\noindent and, therefore, the bound  \eqref{accuracy-H2} can be  represented as
	\begin{equation}\label{accuracy-H2-II}
		\sup_{\theta\in\Theta}\E_\zs{\theta}\big( \theta^{*}_\zs{T}-\theta\big)^2\le
		\frac{1}{I_\zs{a,*} T}+
		\mathrm{O}\left(\frac{1}{T^{\frac{3m}{2m+1}}}\right)
		\quad\mbox{as}\quad T\to\infty \,.
	\end{equation}

\end{remark}

\noindent
Now, we consider the estimation problem for the parameter $a$ in \eqref{sec:Intr.CIR} when the coefficient $b$ is known, i.e. $\theta=a$.
In this case the Maximum Likelihood estimator is given as
\begin{equation}\label{sec:SqPrs-1-MLE-a-2}
	\wh{\theta}_\zs{T}=\frac{b T+\int^{T}_\zs{0}\,X^{-1}_\zs{t}\d X_\zs{t}}{\int^{T}_\zs{0}\, X^{-1}_\zs{t}\d t}	
	\,.
\end{equation}

\noindent  Similarly to 
\eqref{eq:SMLE-bbEs-1} we define
the  sequential estimation  procedure $\delta_\zs{H}=(\tau_\zs{H}\,,\,\wh{\theta}_\zs{\tau_\zs{H}})$ with $H>0$ for the parameter $\theta$ as
\begin{equation}\label{eq:SMLE-a}
	\tau_\zs{H}=\inf{\left(t:\int_\zs{0}^{t}\,X^{-1}_\zs{s}\d s \ge H\right)}
	\quad\mbox{and}\quad
	\wh{\theta}_\zs{\tau_\zs{H}}=\frac{b\tau_\zs{H}+ \int_\zs{0}^{\tau_\zs{H}}\,X^{-1}_\zs{s}\d X_\zs{s} }{H}\,.
\end{equation}
\noindent An extension from this result is that we can define the following truncated sequential procedure $\wt{\delta}_\zs{T}=(\wt{\tau}_\zs{H,T},\wt{\theta}_\zs{H,T})$ where the stopping time $\wt{\tau}_\zs{H,T}$ and the associated estimator $\wt{\theta}_\zs{H,T}$ are defined by 
\begin{equation}\label{eq:TrunEst-a}
	\wt{\tau}_\zs{H,T}=\tau_\zs{H}\wedge T\quad\textrm{and}\quad\wt{\theta}_\zs{H,T}=\wh{\theta}_\zs{\tau_\zs{H}}\ \Chi_{\{\tau_\zs{H}\le T\}}.
\end{equation}
\noindent 
We need the following integral
\begin{align}\label{m2}
	\mu_\zs{a,\theta}
	=
	\int_\zs{\bbr_\zs{+}}\,\varphi(z)\,
	\q_\zs{\theta,b}(z)\d z
	\,,\quad 
	\varphi(x)=\min(x^{-1},\r)
\end{align}
\noindent  
where 
the density $\q_\zs{\theta,b}$ is defined in
\eqref{ergdenssityCIR} and $\r\ge 1$ is some threshold which will be specified later. 
For any compact {$\Theta\subset (\sigma/2,+\infty)$} we set
\begin{equation}
	\label{sec:Para-1-mu-2}
	\mu_\zs{a,*}
	=\inf_{\theta\in\Theta}\,
	\mu_\zs{a,\theta}\,,
	\quad
	\a_\zs{min}=\min_\zs{\theta\in \Theta} \theta
	\quad\mbox{and}\quad
	\a_\zs{max}=\max_\zs{\theta\in \Theta} \theta
	\,.
\end{equation}

\noindent  Similarly to Theorem~\ref{Th.thm:b-1}, we study the non-asymptotic properties of the procedure \eqref{eq:TrunEst-a}.

\begin{theorem}\label{Th.thm:est-a-1}
	For any  $T>0$, $0<H<\mu_\zs{a,*}T$  and compact {$\Theta\subset (\sigma/2,+\infty)$}
	the sequential procedure \eqref{eq:TrunEst-a}  for any $m\ge 2$
	possesses  the following non-asymptotic mean square accuracy
	\begin{equation}
		\label{ineq:1}
		\sup_{\theta\in\Theta}\E_\zs{\theta}(	\wt{\theta}_\zs{H,T}-\theta)^2\le \frac{\sigma}{H}+\frac{T^m \r^{2m} \V_\zs{m}}{(\mu_\zs{a,*}T-H)^{2m}}
		:=\e_\zs{m}(H,T)		
		\,,
	\end{equation}
	\noindent where 
	$$
	\V_\zs{m}=\frac{\a_\zs{max}^2 \L_\zs{m}}{\sigma^{2m}}\left(\frac{4\,e^{\beta }}{\alpha_\zs{min}}+\frac{2^{\alpha_\zs{max}}\Gamma_\zs{max}}{ \beta^{\alpha_\zs{min}}\wedge\beta^{\alpha_\zs{max}}}+\,
	\frac{{2^{\alpha_\zs{max}}}\,}{ \beta}\right)^{2m}\,,
$$
	\noindent  in which $\alpha_\zs{min}=2\a_\zs{min}/\sigma$, $\alpha_\zs{max}=2\a_\zs{max}/\sigma$,  $\Gamma_\zs{max}=\max_\zs{\alpha_\zs{min}\le \alpha\le \alpha_\zs{max}}\,
	\Gamma(\alpha)$ and $\L_\zs{m}$ is defined in \eqref{sec:LL-m-1}.
\end{theorem}

\begin{proof}
	\noindent  First, note that
	from \cite[Theorem 1]{BenAlayaetal2025} it follows that
	\begin{equation}\label{ineq:varalpha}
		\E_\zs{\theta}(	\wt{\theta}_\zs{H,T}-\theta)^2\le \E_\zs{\theta}(	\wh{\theta}_\zs{\tau_\zs{H}}-\theta)^2\Chi_{\{\tau_\zs{H}\le T\}}+\theta^2\P_\theta(\tau_\zs{H}> T)\le \frac{\sigma}{H}+\theta^2\P_\theta(\tau_\zs{H}> T)\,.
	\end{equation}
	\noindent  To estimate the last term in this inequality note that $\varphi(x)=\min\big(x^{-1},\r \big) \le x^{-1}$ and, therefore, we can deduce that 
	\begin{equation}
		\label{ineq:prob-tau2}
		\P_\zs{\theta} (\tau_\zs{H}>T)=
		\P_\zs{\theta}\left(\int^{T}_\zs{0} X_\zs{s}^{-1}\d s<H\right)
		\le 
		\P_\zs{\theta}\left(\int^{T}_\zs{0} \varphi(X_\zs{s})\d s<H\right)\,.
	\end{equation}
	\noindent   Now,  to use the deviations in the ergodic theorem for the process  \eqref{sec:Intr.CIR} we set
	\begin{equation}
		\label{Deviation-2}
		\Delta_\zs{T}(\varphi)=\int^{T}_\zs{0}\,\left(
		\varphi(X_\zs{s})
		-
		\mu_\zs{a,\theta}
		\right)\,\d s\,.
	\end{equation}
	\noindent 
	Then, from \eqref{ineq:prob-tau2} and \eqref{sec:CoIn-2-1}  for 
	$0<H<\mu_\zs{a,*}T$ and $m>1$, we have
	\begin{align*}
		\P_\zs{\theta} (\tau_\zs{H}>T)
		&\le \P_\zs{\theta}\big( \mu_\zs{a,\theta} T+	\Delta_\zs{T}(\varphi)<H\big)
		\\[2mm] 
		&\le \P_\zs{\theta} (	|\Delta_\zs{T}(\varphi)|>\mu_\zs{a,*} T-H)
		\le \frac{\E_\zs{\theta}\, \Delta_\zs{T}(\varphi)^{2m}}{(\mu_\zs{a,*} T-H)^{2m}}
		\\[2mm]	
		&\le 
		\frac{T^{m}}{(\mu_\zs{a,*}T-H)^{2m}}\frac{\r^{2m} \L_\zs{m}}{\sigma^{2m}}
		\left(\frac{4\,e^{\beta }}{\alpha}+\frac{2^{\alpha}\Gamma(\alpha)}{ \beta^{\alpha}}
		+\,
		\frac{2^{\alpha}\,}{ \beta}\right)^{2m}.
	\end{align*}
	\noindent  Therefore, 
	\begin{equation}
		\label{sec-aa-tau-aa-bnd}
		\sup_\zs{\theta\in\Theta}\,
		\P_\zs{\theta} (\tau_\zs{H}>T)
		\le \frac{T^{m}\r^{2m} \L_\zs{m}}{(\mu_\zs{a,*}T-H)^{2m}\sigma^{2m}}
		\,
		\left(\frac{4\,e^{\beta }}{\alpha_\zs{min}}+\frac{2^{\alpha_\zs{max}}\Gamma_\zs{max}}{ \beta^{\alpha_\zs{min}}\wedge\beta^{\alpha_\zs{max}}}+\,
		\frac{{2^{\alpha_\zs{max}}}\,}{ \beta}\right)^{2m}\,,
	\end{equation}
	\noindent 
	where the parameters $\alpha_\zs{min}$,  $\alpha_\zs{max}$ and $\Gamma_\zs{max}$ are defined in  \eqref{ineq:1}.
	Using this bound 
	in \eqref{ineq:varalpha},  
	we obtain \eqref{ineq:1}. 
\end{proof}

\noindent Now  similarly to the definition
\eqref{sec;Opt-H} we choose an optimal value for the parameter $H$ 
to minimise the estimation accuracy \eqref{ineq:1}, i.e.
\begin{equation}
	\label{sec;Opt-H-a}
	H^{*}_\zs{T}=\argmin_\zs{0<H< \mu_\zs{a,*} T} \e_\zs{m}(H,T)\,.
\end{equation}
\noindent  Using this parameter in  \eqref{eq:TrunEst-a}, we obtain the following sequential estimation 
procedure $\big(\tau^{*}_\zs{T},\theta^{*}_\zs{T} \big)$, in which
\begin{equation}
	\label{sec;Opt-H-Seq-Pr-a}
	\tau^{*}_\zs{T}=\wt{\tau}_\zs{H^{*}_\zs{T},T} 
	\quad\mbox{and}\quad
	\theta^{*}_\zs{T}=
	\wt{\theta}_\zs{H^{*}_\zs{T},T}
	\,.
\end{equation}

\noindent
Now,
we can show the following result.
\begin{corollary}\label{cor:b--aa}
	Assume that for some $0<\delta<1/2$
	\begin{equation}
		\label{sec:Cond-1}
		\r=\mathrm{O}(T^{\delta})
		\quad\mbox{as}\quad
		T\to\infty\,.
	\end{equation}
	\noindent Then, for  any $m>(1-2\delta )^{-1}$ 
	the optimal truncated procedure  
	\eqref{sec;Opt-H-Seq-Pr-a}
	posses the following asymptotic 
	properties:.
	
	\begin{enumerate}
		
		\item 
		the optimal parameter
		\eqref{sec;Opt-H-a}
		is represented as
		\begin{equation}\label{H1--a}
			H^{*}_\zs{T}=\mu_\zs{a,*} T-
			\r^{\frac{2m}{2m+1}}\,
			(2m\V_\zs{m}{\mu_\zs{a,*}^2}/\sigma)^{\frac{1}{2m+1}} T^{\frac{2+m}{2m+1}}(1+\mathrm{o}(1))
			\quad\mbox{as}\quad T\to\infty \,;
		\end{equation} 
		\item
		\noindent 
		the corresponding optimal estimation accuracy 
		has the following form 
		\begin{equation}\label{accuracy-H2--aa}
			\sup_{\theta\in\Theta}\E_\zs{\theta}\big( \theta^{*}_\zs{T}-\theta\big)^2\le
			\e_\zs{m}(H^{*}_\zs{T},T)=
			\frac{\sigma}{\mu_\zs{a,*} T}+\mathrm{o}\left(\frac{1}{T}\right)
			\quad\mbox{as}\quad T\to\infty \,.
		\end{equation}
	\end{enumerate}
\end{corollary}
\begin{proof}
	\noindent First note that to calculate  the parameter \eqref{sec;Opt-H-a}  one needs to 
	study the  equation
	$$
	\frac{\partial }{\partial H}\,
	\e_\zs{m}(H,T)
	=0
	\,.
	$$ 
	\noindent 
	Using the form of the function $\e_\zs{m}(H,T)$
	defined
	in
	\eqref{ineq:1} the root of this equation can be represented as
	\begin{equation}\label{eq:H1--1-a}
		H^{*}_\zs{T}=\mu_\zs{a,*}T-
		\r^{{\frac{2m}{2m+1}}}\,
		(2m \V_\zs{m}/\sigma)^{\frac{1}{2m+1}}\,
		(H^{*}_\zs{T})^{\frac{2}{2m+1}}
		{T^{\frac{m}{2m+1}}}.
	\end{equation}
	\noindent Taking into account here that $H^{*}_\zs{T}<\mu_\zs{a,*}T$ the parameter $H^{*}_\zs{T}$ 
	can be estimated from below as
	\begin{equation} \label{eq:H1-1-a}
		H^{*}_\zs{T}>\mu_\zs{a,*}T-
		\r^{{\frac{2m}{2m+1}}}\, 
		(2m\V_\zs{m}{\mu_\zs{a,*}^2}/\sigma)^{\frac{1}{2m+1}}T^{\frac{2+m}{2m+1}} 
= \mu_\zs{a,*}T\,\check{\varpi}_\zs{T,m}\,,
	\end{equation}
	\noindent where $\check{\varpi}_\zs{T,m}=1-
		\r^{{\frac{2m}{2m+1}}}\,
		(2m\V_\zs{m}{\mu_\zs{a,*}^{1-2m}}/\sigma)^{\frac{1}{2m+1}}T^{-\frac{m-1}{2m+1}}$.
	Moreover,  using this bound  in \eqref{eq:H1--1-a}
	the parameter $H^{*}_\zs{T}$  can be estimated from above as  
	\begin{equation}\label{eq:H1-2-a}
		H^{*}_\zs{T}<\mu_\zs{a,*}T-
		\r^{{\frac{2m}{2m+1}}}\,
		(2m\V_\zs{m}{\mu_\zs{a,*}^2}/\sigma)^{\frac{1}{2m+1}}\,
	T^{\frac{2+m}{2m+1}}
		\left(
		\check{\varpi}_\zs{T,m}
		\right)^{\frac{2}{2m+1}}\,.
	\end{equation}
	\noindent
	Taking into account 
	the condition \eqref{sec:Cond-1}
	and using the bounds \eqref{eq:H1-1-a} and \eqref{eq:H1-2-a}
	 we get the asymptotic equality \eqref{H1--a}.
	Moreover, 
	using this form in
	the bound
	\eqref{ineq:1}, we obtain  the representation \eqref{accuracy-H2--aa}.
\end{proof}

\begin{remark} 
	\label{Re.sec:non-as--11}
	It should be noted that we can estimate the parameter $\mu_\zs{a,*}$ from below.  First of all note that for any
	$\theta \in\Theta$
	$$
	\mu_\zs{a,\theta}
	=\frac{\beta^{\alpha}}{\Gamma(\alpha)}
	\int^{\infty}_\zs{0}\,\min\big(z^{-1},\r\big)
	\,z^{\alpha-1}\,e^{- \beta z}
	\,
	\d z
	\le
	\frac{\beta^{\alpha}}{\Gamma(\alpha)}
	\int^{\infty}_\zs{0}\,
	\,z^{\alpha-2}\,e^{- \beta z}
	\,
	\d z 
	= 
	\frac{2b}{2\theta-\sigma}
	$$
	\noindent and, therefore,
	$$
	\mu_\zs{a,*}\le 
	\frac{2b}{2\a_\zs{max}-\sigma}\,.
	$$
	\noindent Moreover, note also that 
	we can deduce the following inequality
	\begin{align*}
		\mu_\zs{a,\theta}
		&
		\ge
		\frac{\beta^{\alpha}}{\Gamma(\alpha)}
		\int^{\infty}_\zs{\r^{-1}}
		\,z^{\alpha-2}
		\,
		\d z
		=
		\frac{2b}{2\theta-\sigma}
		-
		\frac{\beta^{\alpha}}{\Gamma(\alpha)}
		\int^{\r^{-1}}_\zs{0}
		\,z^{\alpha-2}\,e^{- \beta z} \d z
		\\[2mm]&
		\ge 
		\frac{2b}{2\theta-\sigma}
		-
		\frac{\beta^{\alpha}}{\Gamma(\alpha)}
		\int^{\r^{-1}}_\zs{0}
		\,z^{\alpha-2}\,
		\,
		\d z
		\ge \frac{2b}{2\a_\zs{max}-\sigma}
		-
		\frac{\u_\zs{*}}{\r^{\alpha_\zs{min}-1}}
		\,,
	\end{align*}
	\noindent in which 
	$$
	\u_\zs{*}
	=
	\frac{ \beta^{\alpha_\zs{min}}\vee \beta^{\alpha_\zs{max}}}{\Gamma_\zs{min} (\alpha_\zs{min}-1)}\,,
	$$
	\noindent  where
	$a\vee b =\max(a,b)$ and $\Gamma_\zs{min}=\min_\zs{\alpha_\zs{min}\le\alpha\le \alpha_\zs{max}}\Gamma(\alpha)$. Now 
	setting here
	for any
	$0<\epsilon<1 \wedge (2\a_\zs{max}-\sigma)\u_\zs{*} / (2b)$
	the threshold $\r$ as
	\begin{equation}
		\label{sec:radius-rr}
		\r= 
		\left(
		\frac{ (2\a_\zs{max}-\sigma)\u_\zs{*}}{2b\epsilon}
		\right)^{\frac{1}{\alpha_\zs{min}-1}}\,,
	\end{equation}
	\noindent we obtain that
	$$
	\mu_\zs{a,*}\ge (1-\epsilon) \frac{2b}{2\a_\zs{max}-\sigma}\,.
	$$
	
	\noindent Therefore, choosing $\epsilon=T^{- \delta (\alpha_\zs{min}-1)}$ for some $0<\delta<1/2$ we obtain that for $T\to \infty$
	$$
	\mu_\zs{a,*}\to
	\frac{2b}{2\a_\zs{max}-\sigma}
	\quad\mbox{and}\quad
	\r= \OO(T^{\delta})
	\,.
	$$
	\noindent Note that the Fisher information in this case
	\begin{equation}
		\label{sec:Fisher-Info--1-bb-0}
		I_\zs{b}(\theta)=\frac{2b}{\sigma (2\theta-\sigma)}
		\quad\mbox{and}\quad
		I_\zs{b,*}=\min_\zs{\theta\in\Theta}I(\theta)=
		I_\zs{b}(\a_\zs{max})=\frac{2b}{\sigma(2\a_\zs{max}-\sigma)}
		\,.
	\end{equation}
	
	\noindent Therefore, from 
	\eqref{H1--a} and
	\eqref{accuracy-H2--aa} it follows immediately that
	
	\begin{equation}\label{H1--a-opt}
		H^{*}_\zs{T}=
		\frac{2b}{2\a_\zs{max}-\sigma} 
		T +\oo(T)
		\quad\mbox{as}\quad T\to\infty
	\end{equation}
	\noindent and
	\begin{equation}\label{accuracy-H2--opt-a}
		\sup_{\theta\in\Theta}\E_\zs{\theta}\big( \theta^{*}_\zs{T}-\theta\big)^2\le
		\frac{1}{I_\zs{b,*} T}+\mathrm{o}\left(\frac{1}{T}\right)
		\quad\mbox{as}\quad T\to\infty \,.
	\end{equation}
\end{remark}

\begin{remark} 
	\label{Re.sec:non-as--22}
	Note that  we will see later that
asymptotically, as  $T\to\infty$,	
	 the bounds  \eqref{accuracy-H2-II}
	and \eqref{accuracy-H2--opt-a} are minimal.
\end{remark}

\subsection{Optimality properties for the procedure \eqref{sec;Opt-H-Seq-Pr}.}

\noindent
Now let us consider the properties of stopping time that determines the duration of  estimation in procedure
\eqref{sec;Opt-H-Seq-Pr}.

\begin{proposition}
	\label{Pr.sec:SeqEst-1}
	For  any compact set $\Theta\subset ]0,+\infty[$ the stopping time $\tau^{*}_\zs{T}$ defined in
	the  procedure \eqref{sec;Opt-H-Seq-Pr} for any $r>0$
	satisfies the following asymptotic property  
	\begin{equation}
		\label{sec:Sqe-MeanTime-1bis}
		\lim_\zs{T\to\infty}\sup_\zs{\theta\in \Theta}\E_\zs{\theta}\left| \frac{\tau^{*}_\zs{T}}{T}-
		\frac{\theta }{\b_\zs{max}}\right|^{r}=0\,.
	\end{equation}
\end{proposition}
\begin{proof}
	First note that for the stopping time 
	\eqref{eq:SMLE-bbEs-1} 
	in \cite{BenAlayaetal2025} it is shown that
	any compact set $\Theta\subset ]0,+\infty[$ and any $r>0$
	\begin{equation}
		\label{sec:Sqe-MeanTime-1}
		\lim_\zs{H\to\infty}\sup_\zs{\theta\in \Theta}\E_\zs{\theta}\left| \frac{\tau_\zs{H}}{H}-\frac{\theta}{a}\right|^{r}=0
		\,.
	\end{equation}
	
	\noindent 
	Note also that using the bound 
	\eqref{sec:PP-tau-Up} and the representation \eqref{H1}
	one can obtain that for any $m\ge 2$
	\begin{equation}
		\label{sec:Sqe-MeanTime-**-1}
	{	\P_\zs{\theta}} (\tau_\zs{H^{*}_\zs{T}}> T)=
		\mathrm{O}\left(\frac{1}{T^{\frac{3m}{2m+1}}}\right)
		\quad\mbox{as}\quad T\to\infty \,.
	\end{equation}
	
	\noindent Now using this we obtain that for any $\theta$ from $\Theta$
	
	\begin{align*}
		\E_\zs{\theta}\left| \frac{\tau^{*}_\zs{T}}{T}-\frac{\theta}{\b_\zs{max}} \right|^{r}&= 
		\E_\zs{\theta}\left|  \frac{\tau_\zs{H^{*}_\zs{T}}}{T}-\frac{\theta}{\b_\zs{max}}\right|^{r}
		\Chi_\zs{\{\tau_\zs{H^{*}_\zs{T}} \le T\}}+
		\E_\zs{\theta}\left| 1
		-\frac{\theta}{\b_\zs{max}} \right|^{r} \Chi_{\{\tau_\zs{H^{*}_\zs{T}} > T\}}\\[2mm]
		&\le \E_\zs{\theta}\left| \frac{\tau_\zs{H^{*}_\zs{T}}}{T}-\frac{\theta}{\b_\zs{max}}\right|^{r} + {	\P_\zs{\theta}}(\tau_\zs{H^{*}_\zs{T}}> T)\,.
	\end{align*}
	Then, the asymptotic equalities \eqref{H1}, \eqref{sec:Sqe-MeanTime-1} and \eqref{sec:Sqe-MeanTime-**-1} yield the property
	\eqref{sec:Sqe-MeanTime-1bis}.
\end{proof}

\medskip

\noindent Now to study local optimality properties for sequential procedures 
we set  the local class of sequential procedures  defined as
\begin{equation}
	\label{sec:AsOpt-1}
	\cH_\zs{T}(\theta_\zs{0},\gamma)=\left\{
	(\tau\,,\,\wh{\theta}_\zs{\tau})\,:\,
	\sup_\zs{|\theta-\theta_\zs{0}|<\gamma}\,\E_\zs{\theta}\tau\le T
	\right\}\,,
\end{equation}
\noindent where  $\theta_\zs{0}\in\Theta$ and $\gamma>0$ such that $\{|\theta-\theta_\zs{0}|\le \gamma\}\subseteq \Theta$.

\begin{theorem}
	\label{Th.sec:SeqEst-1}
	For any  $\theta_\zs{0}>0$ the sequential procedure \eqref{sec;Opt-H-Seq-Pr} is pointwise optimal
	\begin{equation}
		\label{sec:OPT-EstPrs-1}
		\lim_\zs{\gamma\to 0}\,
		\underline{\lim}_\zs{T\to\infty}
		\frac{
			\inf_\zs{
				(\tau,\wh{\theta}_\zs{\tau})
				\in
				\cH_\zs{T}(\theta_\zs{0},\gamma)}
			\sup_\zs{|\theta-\theta_\zs{0}|<\gamma}\,
			\E_\zs{\theta}\big( \wh{\theta}_\zs{\tau}-\theta\big)^2}{\sup_\zs{|\theta-\theta_\zs{0}|<\gamma}\,
			\E_\zs{\theta}\big( \theta^{*}_\zs{T}-\theta\big)^2}
		=1\,.
	\end{equation}			
\end{theorem}

\begin{proof}
	First note that as it is established in  \cite[Theorem 5.2.]{BenAlayaetal2025} the process \eqref{sec:Intr.CIR}
	for  any $\theta>0$
	satisfies the LAN condition with the  normalised coefficient $\sqrt{T} I_\zs{a}(\theta)$ defined in \eqref{sec:Fisher-Inf-1}. 
	Therefore, in view of Proposition \ref{Pr.sec:AsOp.-1}  for $k=1$
	we obtain that for any $\theta_\zs{0}>0$ and
	for any  $0<\gamma<\theta_\zs{0}$     
	\begin{equation}
		\label{sec:AsOpt-2-bb}
		\underline{\lim}_\zs{T\to\infty}
		\inf_\zs{(\tau,\wh{\theta}_\zs{\tau})\in\cH_\zs{T}(\theta_\zs{0},\gamma)}\,T
		\sup_\zs{| \theta-\theta_\zs{0}|<\gamma}\E_\zs{\theta}\,(\wh{\theta}_\zs{\tau}-\theta)^{2}
		\ge \frac{1}{I_\zs{a}(\theta_\zs{0})}\,.
	\end{equation}
	\noindent Moreover, it should be noted also that
	from the property \eqref{accuracy-H2-II}
	it follows that
	\begin{align}\nonumber
		\overline{\lim}_\zs{T\to\infty}\,T\,
		&\sup_{\theta\in\Theta}\,
		I_\zs{a}(\theta)\,
		\E_\zs{\theta}\big( \theta^{*}_\zs{T}-\theta\big)^2
		\le
		\max_\zs{\theta\in\Theta} I_\zs{a}(\theta)\,
		\overline{\lim}_\zs{T\to\infty}\,T\,
		\sup_{\theta\in\Theta}\,
		\E_\zs{\theta}\big( \theta^{*}_\zs{T}-\theta\big)^2\\[2mm] \label{accuracy-H2--opt-b--11}
		&\le
		\frac{\max_\zs{\theta\in\Theta} I_\zs{a}(\theta)}{I_\zs{a,*}}
		=
		\frac{\b_\zs{max}}{\b_\zs{min}}\,.
	\end{align}		
	
	\noindent Choosing here $\Theta=[\theta_\zs{0}-\gamma\,,\,\theta_\zs{0}{+}\gamma]$ for $0<\gamma<\theta_\zs{0}$ 
	and taking into account that for this set 
	$\b_\zs{max}/\b_\zs{min}\to 1$ as $\gamma\to 0$,
	we obtain that
	\begin{equation}
		\label{sec:Sqe-EstPrs-1}
		\overline{\lim}_\zs{\gamma\to 0}
		\overline{\lim}_\zs{T\to\infty}\,T\,
		\sup_\zs{|\theta-\theta_\zs{0}|<\gamma}
		I_\zs{a}(\theta)\,
		\E_\zs{\theta}\big( \theta^{*}_\zs{T}-\theta\big)^2\le
		1\,.
	\end{equation}
	\noindent 
	{Now, note} 
	that  the function $I_\zs{a}(\theta)$ is continuous. Therefore, 
	for any $0<\epsilon<1$ there exists some constant $\gamma_\zs{0}>0$  such that
	for all $0<\gamma<\gamma_\zs{0}$
	$$
	I_\zs{a,*}=
	\min_\zs{|\theta-\theta_\zs{0}|<\gamma} I_\zs{a}(\theta)\ge (1-\epsilon) I_\zs{a}(\theta_\zs{0})\,.
	$$
	\noindent  Therefore, for such $\gamma>0$
	\begin{align*}
		\overline{\lim}_\zs{T\to\infty}\,T\,
		&
		\sup_\zs{|\theta-\theta_\zs{0}|<\gamma}\,
		\E_\zs{\theta}\big( \theta^{*}_\zs{T}-\theta\big)^2
		=
		\frac{
			\overline{\lim}_\zs{T\to\infty}\,T\,
			\sup_\zs{|\theta-\theta_\zs{0}|<\gamma}\,I_\zs{a,*}
			\E_\zs{\theta}\big( \theta^{*}_\zs{T}-\theta\big)^2}{I_\zs{a,*}}
		\\[2mm]
		&
		\le
		\frac{
			\overline{\lim}_\zs{T\to\infty}\,T\,
			\sup_\zs{|\theta-\theta_\zs{0}|<\gamma}\,I_\zs{a}(\theta)
			\E_\zs{\theta}\big( \theta^{*}_\zs{T}-\theta\big)^2}{(1-\epsilon)I_\zs{a}(\theta_\zs{0})}
		\le
		\frac{1}{(1-\epsilon)I_\zs{a}(\theta_\zs{0})}\,.
	\end{align*}
	\noindent Taking here the limit as $\gamma\to 0$ and then as $\epsilon\to 0$ we can conclude that
	\begin{equation}
		\label{sec:UpBnd-bb-901}
		\overline{\lim}_\zs{\gamma\to 0}
		\overline{\lim}_\zs{T\to\infty}\,T\,
		\sup_\zs{|\theta-\theta_\zs{0}|<\gamma}\,
		\E_\zs{\theta}\big( \theta^{*}_\zs{T}-\theta\big)^2
		\le
		\frac{1}{I_\zs{a}(\theta_\zs{0})}\,.
	\end{equation}
	\noindent This implies the equality
	\eqref{sec:OPT-EstPrs-1}. Hence Theorem  \ref{Th.sec:SeqEst-1}.
\end{proof}

\noindent
To study optimality properties over some arbitrairy compact set $\Theta$ we use the  the same way as  
in  \cite{BenAlayaetal2025}, i.e.
for some family of
sequential procedures $\big({\tau}^{*}_\zs{T},\theta^{*}_\zs{T}\big)_\zs{T>0}$
such that for any parameter $\theta\in\Theta$  the expectation
$\E_\zs{\theta}\,\tau^{*}_\zs{T}\to+\infty$ as $T\to\infty$  we use the following class 
\begin{equation}
	\label{sec;AssOpt-GlCls-1bis}
	\Xi^{*}_\zs{T}
	=\left\{(\tau, \wh{\theta}_\zs{\tau})\,:\,
	\sup_\zs{\theta\in\Theta}\,
	\frac{\E_\zs{\theta}\tau}{\E_\zs{\theta} \tau^{*}_\zs{T}}
	\le 1
	\right\}\,.
\end{equation}
\noindent 
\begin{theorem}\label{Th.sec:AsOp.-2}
	For any compact set $\Theta\subset ]0,+\infty[$,  the sequential procedure
	\eqref{sec;Opt-H-Seq-Pr}	
	is asymptotically optimal in the minimax setting, i.e.
	\begin{equation}
		\label{sec:AsOpt-SeqPr-1-b-!!}
		\lim_\zs{T\to\infty}
		\frac{\inf_\zs{(\tau,\wh{\theta_\zs{\tau}}) \in\Xi^{*}_\zs{T}}\,
			\sup_\zs{ \theta\in\Theta}\,\E_\zs{\theta}\,\,(\wh{\theta}_\zs{\tau}-\theta)^2}
		{\sup_\zs{ \theta\in\Theta}\,\E_\zs{\theta}\,(\theta^{*}_\zs{T}-\theta)^2}=1\,,
	\end{equation}
	\noindent where the class $\Xi^{*}_\zs{T}$ is defined in \eqref{sec;AssOpt-GlCls-1bis}  through the stopping time $\tau^{*}_\zs{T}$ introduced in	
	\eqref{sec;Opt-H-Seq-Pr}.	
\end{theorem}

\begin{proof}
	First note that the property \eqref{sec:Sqe-MeanTime-1bis} implies that for any $\theta\in\Theta$
	the expectation
	$\E_\zs{\theta} \tau^{*}_\zs{T}\to\infty$ as $T\to\infty$. 
	It should be also added that thanks to the property
	\eqref{sec:Sqe-MeanTime-1bis}
	the condition $\C_\zs{1})$ in Section \ref{sec:App-0} holds true. Moreover, the condition $\C_\zs{2})$
	for this case is established in
	\cite[Theorem 5.2.]{BenAlayaetal2025}. Therefore,  Theorem  \ref{Th.sec:AsOp.-1bis} with $k=1$ from Appendix
	implies that
	$$
	\lim_\zs{T\to\infty}
	\inf_\zs{(\tau, \wh{\theta}_\zs{\tau})\in\Xi^{*}_\zs{T}}\,
	\sup_\zs{ \theta\in\Theta}\, I_\zs{a}(\theta) \E_\zs{\theta}\tau^{*}_\zs{T}\,
	\E_\zs{\theta}\,(\wh{\theta}_\zs{\tau}-\theta)^{2}
	\ge 1\,.
	$$
	\noindent Again using here  the property \eqref{sec:Sqe-MeanTime-1bis}  and the form of the Fisher information
	defined in  \eqref{sec:Fisher-Inf-1} we obtain that
	$$
	\lim_\zs{T\to\infty}
	\inf_\zs{(\tau, \wh{\theta}_\zs{\tau})\in\Xi^{*}_\zs{T}}\,
	\sup_\zs{ \theta\in\Theta}\,T\,
	\E_\zs{\theta}\,(\wh{\theta}_\zs{\tau}-\theta)^{2}
	\ge \frac{\sigma \b_\zs{max}}{a}
	\,.
	$$
	\noindent  Now the bound \eqref{accuracy-H2} implies the optimality property \eqref{sec:AsOpt-SeqPr-1-b-!!}.
\end{proof}

\begin{remark}\label{Re.sec:OPt--1}
Il should be noted that  Theorem 2.2 and 
Theorem 5.2
 from \cite{BenAlayaetal2025} 
are shown under the condition $a>\sigma/2$. Indeed, these result hold true for any $a>0$.
\end{remark}

\subsection{Optimality properties for the procedure \eqref{sec;Opt-H-Seq-Pr-a}}

Now let us consider the properties of stopping time that determines the duration of  estimation in procedure
\eqref{sec;Opt-H-Seq-Pr-a}.

\begin{proposition}
	\label{Pr.sec:SeqEst-1-a}
	For any fixed $b>0$  any compact set $\Theta\subset ]\sigma/2,+\infty[$ the stopping time $\tau^{*}_\zs{T}$ defined in
	the  procedure  \eqref{sec;Opt-H-Seq-Pr-a} for any $r>0$
	satisfies the following asymptotic property  
	\begin{equation}
		\label{sec:Sqe-MeanTime-1bis-a}
		\lim_\zs{T\to\infty}\sup_\zs{\theta\in \Theta}\E_\zs{\theta}\left| \frac{\tau^{*}_\zs{T}}{T}-
		\frac{2\theta-\sigma}{2\a_\zs{max}-\sigma}
		\right|^{r}=0\,.
	\end{equation}
\end{proposition}
\begin{proof}
	First of all note that for the stopping time defined in
	\eqref{eq:SMLE-a}
	in \cite[Theorem 2.7]{BenAlayaetal2025} it is shown that for 
	any compact set $\Theta\subset ]0,+\infty[$ and any $r>0$
	\begin{equation}
		\label{sec:Sqe-MeanTime-1-a}
		\lim_\zs{H\to\infty}\sup_\zs{\theta\in \Theta}\E_\zs{\theta}\left| \frac{\tau_\zs{H}}{H}-\frac{2\theta-\sigma}{2b}\right|^{r}=0
		\,.
	\end{equation}
	
	\noindent 
	Note also that using the bound 
	\eqref{sec-aa-tau-aa-bnd}
	and the representation \eqref{H1--a}
	one can obtain that
	\begin{equation}
		\label{sec:Sqe-MeanTime-**-1-a}
		{	\P_\zs{\theta}}(\tau_\zs{H^{*}_\zs{T}}> T)=
		\oo\left(\frac{1}{T}\right)
		\quad\mbox{as}\quad T\to\infty \,.
	\end{equation}
	
	\noindent Now using this we obtain that for any $\theta$ from $\Theta$
	
	\begin{align*}
		\E_\zs{\theta}\left| \frac{\tau^{*}_\zs{T}}{T}-
		\frac{2\theta-\sigma}{2\a_\zs{max}-\sigma}
		\right|^{r}&= 
		\E_\zs{\theta}\left|  \frac{\tau_\zs{H^{*}_\zs{T}}}{T}-\frac{2\theta-\sigma}{2\a_\zs{max}-\sigma} \right|^{r}
		\Chi_\zs{\{\tau_\zs{H^{*}_\zs{T}} \le T\}}\\[2mm]
&+
		\E_\zs{\theta}\left| 1
		-\frac{2\theta-\sigma}{2\a_\zs{max}-\sigma} \right|^{r} \Chi_{\{\tau_\zs{H^{*}_\zs{T}} > T\}}\\[2mm]
		&\le \E_\zs{\theta}\left| \frac{\tau_\zs{H^{*}_\zs{T}}}{T}-\frac{2\theta-\sigma}{2\a_\zs{max}-\sigma}\right|^{r} + {	\P_\zs{\theta}}(\tau_\zs{H^{*}_\zs{T}}> T)\,.
	\end{align*}
	Then, the asymptotic equalities \eqref{H1--a-opt}, \eqref{sec:Sqe-MeanTime-1-a}
	and \eqref{sec:Sqe-MeanTime-**-1-a} yield the property
	\eqref{sec:Sqe-MeanTime-1bis-a}.
\end{proof}

\medskip

\begin{theorem}
	\label{Th.sec:SeqEst-1=a}
	For any fixed $b>0$ and  $\theta_\zs{0}>\sigma/2$ the sequential procedure  \eqref{sec;Opt-H-Seq-Pr-a} is point-wise optimal
	\begin{equation}
		\label{sec:OPT-EstPrs-1-a}
		\lim_\zs{\gamma\to 0}\,
		\underline{\lim}_\zs{T\to\infty}
		\frac{
			\inf_\zs{
				(\tau,\wh{\theta}_\zs{\tau})
				\in
				\cH_\zs{T}(\theta_\zs{0},\gamma)}
			\sup_\zs{|\theta-\theta_\zs{0}|<\gamma}\,
			\E_\zs{\theta}\big( \wh{\theta}_\zs{\tau}-\theta\big)^2}{\sup_\zs{|\theta-\theta_\zs{0}|<\gamma}\,
			\E_\zs{\theta}\big( \theta^{*}_\zs{T}-\theta\big)^2}
		=1\,,
	\end{equation}			
	\noindent where the class $\cH_\zs{T}(\theta_\zs{0},\gamma)$ is defined in \eqref{sec:AsOpt-1}.
\end{theorem}

\begin{proof}
	First note that as it is established in  \cite[Theorem 5.3.]{BenAlayaetal2025} the process \eqref{sec:Intr.CIR}
	for fixed $b>0$ and any $\theta>\sigma/2$
	satisfies the LAN condition with the  corresponding Fisher information $I_\zs{b}(\theta)$ defined in \eqref{sec:Fisher-Info--1-bb-0}. 
	Therefore, in view of Proposition \ref{Pr.sec:AsOp.-1}  for $k=1$
	we obtain that for any $\theta_\zs{0}>\sigma/2$ and
	for any  $\sigma/2<\gamma<\theta_\zs{0}$     
	\begin{equation}
		\label{sec:AsOpt-2-aa}
		\underline{\lim}_\zs{T\to\infty}
		\inf_\zs{(\tau,\wh{\theta}_\zs{\tau})\in\cH_\zs{T}(\theta_\zs{0},\gamma)}\,T
		\sup_\zs{| \theta-\theta_\zs{0}|<\gamma}\E_\zs{\theta}\,(\wh{\theta}_\zs{\tau}-\theta)^{2}
		\ge \frac{1}{I_\zs{b}(\theta_\zs{0})}\,.
	\end{equation}
	\noindent Moreover, it should be noted also that similarly to the bound
	\eqref{accuracy-H2--opt-b--11}
	using the inequality
	\eqref{accuracy-H2--opt-a} one can deduce the following upper bound for the mean square accuracy
	\begin{equation}  \label{accuracy-H2--opt-a--11}
		\overline{\lim}_\zs{T\to\infty}\,T\,
		\sup_{\theta\in\Theta}\,
		I_\zs{b}(\theta)\,
		\E_\zs{\theta}\big( \theta^{*}_\zs{T}-\theta\big)^2
		\le
		\frac{\max_\zs{\theta\in\Theta} I_\zs{b}(\theta)}{I_\zs{b,*}}
		=
		\frac{2\a_\zs{max}-\sigma}{2\a_\zs{min}-\sigma}\,.
	\end{equation}		
	\noindent Choosing here $\Theta=[\theta_\zs{0}-\gamma\,,\,\theta_\zs{0}{+}\gamma]$ for  $\sigma/2<\gamma<\theta_\zs{0}$
	and taking into account that for this set 
	$\a_\zs{max}/\a_\zs{min}\to 1$ as $\gamma\to 0$,
	we obtain that
	\begin{equation}
		\label{sec:Sqe-EstPrs-1-aa}
		\overline{\lim}_\zs{\gamma\to 0}
		\overline{\lim}_\zs{T\to\infty}\,T\,
		\sup_\zs{|\theta-\theta_\zs{0}|<\gamma}
		I_\zs{b}(\theta)\,
		\E_\zs{\theta}\big( \theta^{*}_\zs{T}-\theta\big)^2\le
		1\,.
	\end{equation}
	\noindent 
	Now from here through  the same way used in the inequality 
	\eqref{sec:UpBnd-bb-901} we obtain that
	$$
	\overline{\lim}_\zs{\gamma\to 0}
	\overline{\lim}_\zs{T\to\infty}\,T\,
	\sup_\zs{|\theta-\theta_\zs{0}|<\gamma}\,
	\E_\zs{\theta}\big( \theta^{*}_\zs{T}-\theta\big)^2
	\le
	\frac{1}{I_\zs{b}(\theta_\zs{0})}\,.
	$$
	\noindent From  this and the lower bound
	\eqref{sec:AsOpt-2-aa}
	it follows
	the property
	\eqref{sec:OPT-EstPrs-1-a}. Hence Theorem  \ref{Th.sec:SeqEst-1=a}.
\end{proof}

\noindent  Now using Proposition  \ref{Pr.sec:SeqEst-1-a} and the form of the Fisher information $I_\zs{b}$ given in   \eqref{sec:Fisher-Info--1-bb-0}
similarly to Theorem \ref{Th.sec:AsOp.-2} one can show the following result.
\begin{theorem}\label{Th.sec:AsOp.-2-aa}
	For any $b>0$ and any compact set $\Theta\subset ]\sigma/2,+\infty[$,  the sequential procedure
	\eqref{sec;Opt-H-Seq-Pr-a}	
	is asymptotically optimal in the minimax setting, i.e.
	\begin{equation}
		\label{sec:AsOpt-SeqPr-1-b-22}
		\lim_\zs{T\to\infty}
		\frac{\inf_\zs{(\tau,\wh{\theta_\zs{\tau}}) \in\Xi^{*}_\zs{T}}\,
			\sup_\zs{ \theta\in\Theta}\,\E_\zs{\theta}\,\,(\wh{\theta}_\zs{\tau}-\theta)^2}
		{\sup_\zs{ \theta\in\Theta}\,\E_\zs{\theta}\,(\theta^{*}_\zs{T}-\theta)^2}=1\,,
	\end{equation}
	\noindent where the class $\Xi^{*}_\zs{T}$ is defined in \eqref{sec;AssOpt-GlCls-1bis}  through the stopping time $\tau^{*}_\zs{T}$ introduced in	
	\eqref{sec;Opt-H-Seq-Pr-a}.	
\end{theorem}

\begin{remark}
	\label{Re.sec:opt--1}
	It should be noted that the properties
\eqref{sec:Sqe-MeanTime-1bis} and \eqref{sec:Sqe-MeanTime-1bis-a} imply that	
for $T\to\infty$ the mean observations durations $\E_\zs{\theta} \tau^{*}_\zs{T}-T\to-\infty$ for $\theta\neq \b_\zs{max}$ and
  $\theta\neq \a_\zs{max}$ respectively. This means that  to obtain the optimality properties with respect to the non-sequential estimation based on the fixed observations  duration $T$
	 the procedures  \eqref{sec;Opt-H-Seq-Pr} and  \eqref{sec;Opt-H-Seq-Pr-a} essentially
	  reduce the duration of observations.
\end{remark}

\section{{Two-dimensional} truncated sequential estimation method}
\label{sec:MultPr}

\subsection{Guaranteed estimation}
\noindent
Now we develop a truncated sequential estimation method 
for the two dimension parameter 
$\theta=(a,b)^{\top}$ from some compact $\Theta\subset ]\sigma/2,+\infty[ \times ]0,\infty[$. 
For this problem, it is convenient to represent the process
\eqref{sec:Intr.CIR} 
as
\begin{equation}\label{eq:CIR2}
	\d X_\zs{t}=\g_\zs{t}^{\top}\theta \d t+\sqrt{\sigma X_\zs{t}}\d W_\zs{t}\,,\quad 0\le t\le T\,,
\end{equation}
where $\g_\zs{t}=\left(1\,,\,-X_\zs{t}\right)^{\top}$. Note that 
in view of the results from \cite{BAK2013}
in this case this process is ergodic and one can show that 
$\P_\zs{\theta}$ a.s.
for any $\theta\in ]\sigma/2,+\infty[ \times ]0,\infty[$
the following limit equalities hold true 
\begin{equation}\label{sec:Mlt.erg-prs-1}
	\lim_\zs{t\to\infty}
	\frac{1}{t} \int_\zs{0}^{t} X^{-1}_\zs{s}\d s
	= \frac{2b}{2a-\sigma}
	:=f_\zs{1}
	\quad\mbox{and}\quad
	\lim_\zs{t\to\infty}
	\frac{1}{t} \int_\zs{0}^{t}X_\zs{s}\d s
	=\frac{a}{b}:=f_\zs{2}\,.
\end{equation}
\noindent Therefore, setting
\begin{equation}
	\label{sec:MultPr-1}
	G_\zs{t}=\int_\zs{0}^tX^{-1}_\zs{s}\g_\zs{s}\g_\zs{s}^{\top}\d s=\left(\begin{array}{cc} \int_\zs{0}^tX^{-1}_\zs{s}\d s&-t\\[4mm]
		-t&\int_\zs{0}^t{X_\zs{s}}ds\end{array}\right)\,,
\end{equation}
\noindent we obtain that
\begin{equation}\label{sec:Mlt.cond-1}
	\lim_\zs{t\to\infty}
	\frac{1}{t}G_\zs{t}=F
	=
	\left(
	\begin{array}{cc}
		f_\zs{1} &-1\\[4mm]
		-1& f_\zs{2}
	\end{array}
	\right)
	\qquad\P_\zs{\theta}-\mbox{a.s.,}
\end{equation}

\noindent  where the matrix $F=F(\theta)$ is positive definite.   To estimate the parameters $\theta$ we use the sequential procedure 
introduced in \cite{BenAlayaetal2025}.
To do this  first we will use the family of stopping times $\big(	\t_\zs{z} \big)_\zs{z>0}$ defined as
\begin{equation}
	\label{sec:MultPr-1-00}
	\t_\zs{z}=\inf\left\{t\ge 0:\int_\zs{0}^{t}\,X^{-1}_\zs{s} \, |\g_\zs{s}|^2 \d s\ge z\right\}\,,
\end{equation}
\noindent provided that  $\inf\{\emptyset\}=+\infty$. It should be noted that the properties  \eqref{sec:Mlt.erg-prs-1}
imply directly that $\t_\zs{z}<\infty$ a.s. for any $z>0$. Now for any  non-random sequence
of non-decreasing positive numbers $(\kappa_\zs{n})_\zs{n\ge 1}$ for which
\begin{equation}
	\label{sec:MultPr-1-News-90432}
	\sum_\zs{n\ge 1}\,\frac{1}{\kappa_\zs{n}}<\infty
\end{equation}
\noindent
we define the sequential procedures 
$
\left(\t_\zs{n},\wh{\theta}_\zs{\t_\zs{n}}\right)_\zs{n\ge 1}$ as
\begin{equation}
	\label{sec:MlT-MLE-1}
	\t_\zs{n}=\t_\zs{\kappa_\zs{n}}
	\quad\mbox{and}\quad
	\wh{\theta}_\zs{\t_\zs{n}}=G^{+}_\zs{\t_\zs{n}}\int_\zs{0}^{\t_\zs{n}}\, X^{-1}_\zs{s} \g_\zs{s} \d X_\zs{s}
	\,.	
\end{equation}
\noindent Here the matrix $G^{+}=G^{-1}$ if the inverse matrix $G^{-1}$ exists and $G^{+}=0$ otherwise.  
The number of these estimates required to construct a two-step sequential procedure is determined by the following stopping time
\begin{equation}
	\label{sec:MultPr-b-nn-122}
	\upsilon_\zs{H}
	=
	\inf
	\left\{k\ge1:\sum_{n=1}^k\b_\zs{n}^2\ge H
	\right\}
	\,,
\end{equation}
\noindent where $H$  is a positive non-random threshold that will  be chosen below and

\begin{equation}
	\label{sec:MultPr-b-bnn-101}
	\b_\zs{n}=\frac{1}{| G^{-1}_\zs{\t_\zs{n}}|\, \kappa_\zs{n}}\,\Chi_\zs{\{ \lambda_\zs{min}(G_\zs{\t_\zs{n}})>0\}}
	\,.
\end{equation}
\noindent 
Here $|\cdot|$ denotes the  Euclidean  norm for the vectors and matrices and  $\lambda_\zs{min}(G)$ is the minimal eigenvalue of the matrix $G$. As is shown in \cite{BenAlayaetal2025}
for any $\theta\in ]\sigma/2,+\infty[ \times ]0,\infty[$
\begin{equation}
	\label{sec:MultPr-AsPrs-bb}
	\lim_\zs{n\to\infty} \b^{2}_\zs{n}=\b^{2}_\zs{*}=
	\frac{1}{( |F^{-1}|\,\tr \,F)^{2}}
	>0
	\qquad
	\P_\zs{\theta} -
	\quad\mbox{a.s.}
	\,,
\end{equation}
\noindent 
i.e. $\sum_\zs{n\ge 1} \b^{2}_\zs{n}=+\infty$
and, therefore,  for any $H>0$ the moment  \eqref{sec:MultPr-b-nn-122} is finite a.s.  Moreover,
setting 
\begin{equation}\label{u*}
	\u_\zs{*}=\max_\zs{\theta\in\Theta} \big(|F^{-1}|\, \tr \, F\big)^{2}\,,
\end{equation}

\noindent 
we chose the sequence $(\kappa_\zs{n})_\zs{n\ge 1}$ 
as
\begin{equation}
	\label{sec:seq-kappa}
	\kappa_\zs{n}
	=
	\left\{
	\begin{array}{cc}
		H\,,&\quad\mbox{for}\quad n\le \n^{*}_\zs{H}\,; \\[2mm]
		\kappa^{*}_\zs{n}\,,&
		\quad\mbox{for}\quad n > \n^{*}_\zs{H}\,,
	\end{array}
	\right.
\end{equation}
where $\n^{*}_\zs{H}= 2 \u_\zs{*} H$,
\noindent  and
$(\kappa^{*}_\zs{n})_\zs{n\ge 1}$ is an increasing sequence  such that  for all $n$ it is bounded from below as $\kappa^{*}_\zs{n}\ge n$ and for some
constants  $\varpi>1$ and $0<\delta^{*}<1/2$, 
\begin{equation}
	\label{sec:seq-kappa--0}
	\overline{\lim}_\zs{n\to\infty}\,n^{-\varpi}\,\kappa^{*}_\zs{n}\,<\,\infty
	\quad\mbox{and}\quad
	\overline{\lim}_\zs{n\to\infty}\,n^{-\delta^{*}}
	\sum^{n}_\zs{k=1}\,\frac{1}{\sqrt{\kappa^{*}_\zs{k}}}\,<\,\infty\,.
\end{equation}

\noindent  For example, we can take 
$\kappa^{*}_\zs{n}=n^{\varpi}$ and $\delta^{*}=(2-\varpi)/2$ for some $1<\varpi<2$.

\noindent
Now using this sequence one can show the following property for the moment  $\upsilon_\zs{H}$.
\begin{lemma}\label{Le.sec:App-33}
	For any compact set $\Theta\subset (\sigma/2,+\infty)\, \times\, (0,+\infty)$,  for any $r>1$ and $H>0$
	there exists some constant $\v^{*}_\zs{r}>0$ such that for any
	$n>\u_\zs{*} H$ 
	the distribution tail of the stopping time \eqref{sec:MultPr-b-nn-122} 
	can be estimated from above as
	\begin{equation}\label{sec:uppBnd-tail--0}
		\sup_\zs{\theta\in\Theta}\,
		\P_\zs{\theta}\,
		\left(
		\upsilon_\zs{H}
		>n
		\right)
		\le\,\v^{*}_\zs{r}\,
		\frac{(2\u_{*})^{2r}H^{r}+n^{2\delta^{*} r}}{(n-\u_\zs{*} H)^{2r}}\,,
	\end{equation}
	\noindent where  the costant  $0<\delta^{*}<1/2$ is given in the condition  \eqref{sec:seq-kappa--0}.
\end{lemma}
\noindent The proof of this lemma is given in Appendix.

\noindent
In this case  we define the aggregated sequential estimation procedure $(\tau_\zs{H},\ov{\theta}_\zs{H})$ as

\begin{equation}
	\label{sec:two-step-!}
	\tau_\zs{H}=\t_\zs{\upsilon_\zs{H}}
	\quad\mbox{and}\quad
	\ov{\theta}_\zs{H}
	=
	\left( \sum_{n=1}^{\upsilon_\zs{H}}\,\b_\zs{n}^2\right)^{-1}\, \sum_{n=1}^{\upsilon_\zs{H}}\, \b^{2}_\zs{n}\, \wh{\theta}_\zs{\t_\zs{n}}
	\,.
\end{equation}
\noindent 
In this paper we use the truncated version of this procedure
$(\wt{\tau}_\zs{H,T},\wt{\theta}_\zs{H,T})$ in which
\begin{equation}\label{eq:TrunEst-ab}
	\wt{\tau}_\zs{H,T}=\tau_\zs{H}\wedge T\quad\textrm{and}\quad\wt{\theta}_\zs{H,T}=\ov{\theta}_\zs{H}\ \Chi_{\{\tau_\zs{H}\le T\}}.
\end{equation}

\noindent	 
Now, we study this procedure in non-asymptotic setting, i.e. for arbitraire  fixed $H>0$ and $T\ge  1$. To do this we need the following 
functionals

\begin{equation}\label{muu-03-01}
	\mu_\zs{\theta}=
	\frac{a}{b}
	+	
	\int_\zs{\bbr_\zs{+}}\,\varphi(x)\,
	\q_\zs{\theta}(x)\d x
	\quad\mbox{and}\quad
	\mu_\zs{*}=\inf_{\theta\in\Theta}\, \mu_\zs{\theta}\,,
\end{equation}
where the ergodic density $\q_\zs{\theta}$  is defined in
\eqref{ergdenssityCIR} and the function
$\varphi(x)=\min\big(x^{-1},\r\big)$ in which the threshold $\r\ge 1$ will be chosen later.

\begin{theorem}
	\label{Th.sec:MultPr-343-1}
	For   any compact set $\Theta\subset(\sigma/2,+\infty)\, \times\, (0,+\infty)$, for any duration of observations $T> 1/\mu_\zs{*}$,
for any parameter	$1\le H<\mu_\zs{*}T$	
	and $m\ge 2$  the sequential procedure  \eqref{eq:TrunEst-ab} possesses the fixed guaranteed estimation accuracy, i.e. 
	\begin{equation}
		\label{sec:MultPr--fxd-ac-00}
		\sup_\zs{\theta\in\Theta}\,
		\E_\zs{\theta}\,|\wt{\theta}_\zs{H,T}-\theta|^2
		\,\le
		\frac{(2\u_\zs{*}+\rho^{*}_\zs{H}) \sigma}{H}+\frac{T^m\,\theta_\zs{max} \r^{2m} Z_\zs{m}}{(\mu_\zs{*}T -H)^{2m}}
		+
		\, \frac{2^{4m+1}\v^{*}_\zs{2m}{\theta_\zs{max}}}{H^{2m}}
		\,,
	\end{equation}
	where $\rho^{*}_\zs{H}=\sum_\zs{n> \n^{*}_\zs{H}} \big(\kappa^{*}_\zs{n}\big)^{-1}$,
	$\theta_\zs{max}=\max_\zs{\theta\in\Theta} |\theta|^{2}$, the coefficient $\v^{*}_\zs{2m}$ is given in Lemma \ref{Le.sec:App-33}
	and
	$$
	Z_\zs{m}=
	2^{2m}
	\L_\zs{m}
	\left(\frac{1}{ \b^{2m}_\zs{min}}+
	\frac{1}{\sigma^{2m}}\left(\frac{4\,e^{\beta_\zs{max}}}{\alpha_\zs{min}}+\frac{2^{\alpha_\zs{max}}\Gamma_\zs{max}}
	{\beta_\zs{min}^{\alpha_\zs{min}}\wedge\beta_\zs{min}^{\alpha_\zs{max}}}+\,
	\frac{{2^{\alpha_\zs{max}}}\,}{ \beta_\zs{min}}\right)^{2m}
	\right)
	\,.
	$$	   
\end{theorem}
\begin{proof} First, note that  that on the set $\{\lambda_\zs{min}(G_\zs{\t_\zs{n}}) >0\}$ the sequential MLE  
	\eqref{sec:MlT-MLE-1} can be represented as
	\begin{equation}
		\label{sec:MlSe-98-1}
		\wh{\theta}_\zs{\t_\zs{n}}
		=G_\zs{\t_\zs{n}}^{-1}\int_\zs{0}^{\t_\zs{n}}\,X^{-1}_\zs{s} \g_\zs{s} \d X_\zs{s}=\theta+
		\sqrt{\sigma}\,
		G_\zs{\t_\zs{n}}^{-1}\,
		\eta_\zs{\t_\zs{n}}
		\quad\mbox{and}\quad
		\eta_\zs{\t_\zs{n}}=
		\int_\zs{0}^{\t_\zs{n}}\,
		X^{-1/2}_\zs{s}
		\,
		\g_\zs{s}
		\d W_\zs{s}\,.
	\end{equation}
	
	\noindent Using here the definition \eqref{sec:MultPr-1-00} and the properties of the stochastic integrals we obtain that
	\begin{equation}
		\label{sec:MlSe-98-22}
		\E_\zs{\theta}\, |\eta_\zs{\t_\zs{n}}|^{2}
		=
		\E_\zs{\theta}\,
		\int_\zs{0}^{\t_\zs{n}}\,
		X^{-1}_\zs{s}
		\,
		|\g_\zs{s}|^{2}
		\d s
		=\kappa_\zs{n}\,.
	\end{equation}
	
	\noindent Furthermore, in view of \eqref{sec:MlSe-98-1} we can represent the estimator \eqref{eq:TrunEst-ab} in the following form
	\begin{equation}
		\label{sec:Multi-Par--1098}
		\wt{\theta}_\zs{H,T}
		=\frac{\sum_{n=1}^{\upsilon_\zs{H}}\b^{2}_\zs{n}\,\wh{\theta}_\zs{\t_\zs{n}}}{\sum_{n=1}^{\upsilon_\zs{H}}\b_\zs{n}^2}\ \Chi_\zs{\{\tau_\zs{H}\le T\}}=\left(\theta+
		\sqrt{\sigma}\,
		\frac{\sum_{n=1}^{\upsilon_\zs{H}}\b_\zs{n}\,\xi_\zs{n}}{\sum_{n=1}^{\upsilon_\zs{H}}\b_\zs{n}^2}\right)\Chi_\zs{\{\tau_\zs{H}\le T\}}
	\end{equation}
	\noindent  and	$\xi_\zs{n}=\b_\zs{n}G_\zs{\t_\zs{n}}^{-1}\eta_\zs{\t_\zs{n}}$.	
Note now that
	\begin{align}\label{eq:varalbe}
		\E_\zs{\theta}\,\left|
		\wt{\theta}_\zs{H,T}
		-\theta
		\right|^{2}
		\le\, \E_\zs{\theta}\,\left|
		\ov{\theta}_\zs{H}
		-\theta
		\right|^{2}\ +|\theta|^2\ {	\P_\zs{\theta}}(\tau_\zs{H}> T).
	\end{align}
	\noindent On the one hand, taking into account here the definition
	\eqref{sec:MultPr-b-bnn-101} and the  property \eqref{sec:MlSe-98-22}, we get that
	$$
	\E_\zs{\theta}\, |\xi_\zs{n}|^{2}\le
	\frac{1}{\kappa^{2}_\zs{n}}\,
	\E_\zs{\theta}\, |\eta_\zs{\t_\zs{n}}|^{2}
	=\frac{1}{\kappa_\zs{n}}\,.
	$$
	\noindent Then, from here through the 
	Cauchy-Schwarz-Bunyakovsky inequality and the definition
	\eqref{sec:MultPr-b-nn-122},
	we find
	\begin{align}\label{eq:varalbe-bis}
		\E_\zs{\theta}\,\left|
		\ov{\theta}_\zs{H}
		-\theta
		\right|^{2}
		\le\,
		\sigma\,
		\E_\zs{\theta} 
		\frac{\sum_{n=1}^{\upsilon_\zs{H}}|\xi_\zs{n}|^{2}}{\sum_{n=1}^{\upsilon_\zs{H}}\b_\zs{n}^2}
		\le
		\sigma\,\frac{1}{H}\,
		\sum_\zs{n\ge 1}\,
		\E_\zs{\theta}\,|\xi_\zs{n}|^{2}\le
		\sigma\,\frac{1}{H}\,
		\sum_\zs{n\ge 1}\,\frac{1}{\kappa_\zs{n}}
		\,,
	\end{align}
	\noindent   where in view of the definition  \eqref{sec:seq-kappa}
	$$
	\sum_\zs{n\ge 1}\,\frac{1}{\kappa_\zs{n}}
	\le 
	2\u_\zs{*}
	+
	\sum_\zs{n> \n^{*}_\zs{H}}\,\frac{1}{\kappa^{*}_\zs{n}}
	= 2\u_\zs{*}
	+ \rho^{*}_\zs{H}\,.
	$$   
	\noindent  On the other hand, we have
	\begin{align*}
		\P_\zs{\theta} (\tau_\zs{H}>T)
		&=\P_\zs{\theta} (\t_\zs{\upsilon_\zs{H}}>T)
		=\P_\zs{\theta}\left(\int^{T}_\zs{0} \big( X_\zs{s}+ X^{-1}_\zs{s} \big)\d s<\kappa_\zs{\upsilon_\zs{H}}\right)
		\\[2mm]
		&\le 
		\P_\zs{\theta}\left(\int^{T}_\zs{0}\big( X_\zs{s}+ \varphi(X_\zs{s})\big)\d s<\kappa_\zs{\upsilon_\zs{H}}\right)\,,
	\end{align*}	
	\noindent where $x+\varphi(x)=x+\min(x^{-1},\r )\le x+ x^{-1}$. Therefore,
	\begin{equation}
		\label{ineq:prob-tau3}
		\P_\zs{\theta} (\tau_\zs{H}>T)
		\le 
		\P_\zs{\theta}\left(
		\int^{T}_\zs{0}
		\big(X_\zs{s}+
		\varphi(X_\zs{s})
		\big) 
		\d s<H\right)
		+\P_\zs{\theta}\left({\upsilon_\zs{H}>\n^{*}_\zs{H}}\right)\,.
	\end{equation}
	
	\noindent 
	Now, similarly to 
	\eqref{ineq:prob-tau}
	and
	\eqref{Deviation-2}
	we set
	\begin{equation}
		\label{Deviation-3}
		\D_\zs{T}=
		\int_\zs{0}^T
		\left(
		X_\zs{s}- \frac{a}{b}
		\right) \d s
		\quad\mbox{and}\quad
		\Delta_\zs{T}
		=\int^{T}_\zs{0}\,\left(
		\varphi(X_\zs{s})
		-
		\mu_\zs{1,\theta}
		\right)\,\d s\,,
	\end{equation}
	\noindent where
	$
	\mu_\zs{1,\theta}	=
	\int_\zs{\bbr_\zs{+}}\,\varphi(x)\,
	\q_\zs{\theta}(x)\d x$.
	Using these deviations and the definitions
	\eqref{muu-03-01}  the  integral in \eqref{ineq:prob-tau3}  can be estimated from below 
	for any $\theta\in\Theta$
	as
	$$
	\int^{T}_\zs{0}
	\big(X_\zs{s}+
	\varphi(X_\zs{s})
	\big) 
	\d s
	=\mu_\zs{\theta} T
	+\D_\zs{T}
	+\Delta_\zs{T}
	\ge \mu_\zs{*} T
	+\D_\zs{T}
	+\Delta_\zs{T}\,.
	$$
	
	\noindent
Therefore, the first term in the r.h.s. of \eqref{ineq:prob-tau3} for
	$1<H<\mu_\zs{*}T$ and $\theta\in\Theta$  can be estimated through the Chebyshev's inequality
	for  $m>1$
	as
	\begin{align*}
		\P_\zs{\theta}
		\left(
		\int^{T}_\zs{0}
		\big(X_\zs{s}+
		\varphi(X_\zs{s})
		\big) 
		\d s<H\right)
		&\le
		\P_\zs{\theta}\Big(
		\mu_\zs{*} T
		+\D_\zs{T}
		+\Delta_\zs{T}
		<H\Big)
		\\[2mm]
		&\le 
		\P_\zs{\theta}\Big(
		|\D_\zs{T}|
		+|\Delta_\zs{T}|
		>\mu_\zs{*} T -H\Big)\\[2mm]
		&\le 2^{2m}
		\frac{\E_\zs{\theta}\, |\D_\zs{T}|^{2m}
			+\E_\zs{\theta}\, |\Delta_\zs{T}|^{2m}}{\big( \mu_\zs{*} T -H\big)^{2m}}\,.
	\end{align*}
	
	\noindent Using here inequalities
	\eqref{sec:CoIn-D-T-2-1}
	and \eqref{sec:CoIn-2-1}
	we obtain that
	\begin{align*}
		\sup_\zs{\theta\in\Theta}\,
		\P_\zs{\theta}&
		\left(
		\int^{T}_\zs{0}
		\big(X_\zs{s}+
		\varphi(X_\zs{s})
		\big) 
		\d s<H\right)
		\le 2^{2m}
		\left(
		\frac{1}{ \b^{2m}_\zs{min}}
		+
		\frac{\r^{2m}}{\sigma^{2m}}
		\s^{*}_\zs{m}
		\right)
		\frac{\L_\zs{m}T^{m}}{\big( \mu_\zs{*} T -H\big)^{2m}}\,,
	\end{align*}
	\noindent where
	$$
	\s^{*}_\zs{m} = \sup_\zs{\theta\in\Theta} 
		\left(\frac{4\,e^{\beta }}{\alpha}+\frac{2^{\alpha}\Gamma(\alpha)}{ \beta^{\alpha}}
		+\,
		\frac{2^{\alpha}\,}{ \beta}\right)^{2m}\,.
	$$

	 Taking into account here that $\r\ge 1$ we can obtain that
	
	\begin{equation}
		\label{sec:tail-est-1}
		\P_\zs{\theta}
		\left(
		\int^{T}_\zs{0}
		\big(X_\zs{s}+
		\varphi(X_\zs{s})
		\big) 
		\d s<H\right)
		\le 
		\frac{\r^{2m}\,Z_\zs{m}T^{m}}{\big( \mu_\zs{*} T -H\big)^{2m}}\,,
	\end{equation}
	\noindent where the coefficient $Z_\zs{m}$ is defined in the bound \eqref{sec:MultPr--fxd-ac-00}.
	Considering the second term in the r.h.s. of \eqref{ineq:prob-tau3}, using Lemma \ref{Le.sec:App-33}, for any $r>1$ and $H\ge 1$, we have 
		$$
		\sup_\zs{\theta\in\Theta}\,
		\P_\zs{\theta}\,
		\left(
		\upsilon_\zs{H}
		>\n^{*}_\zs{H}
		\right)
		\le\,
		\v^{*}_\zs{r}\,
		\frac{(2\u_{*})^{2r}H^{r}+{\n^{*}_\zs{H}}^{2\delta^{*} r}}{(\n^{*}_\zs{H} -\u_\zs{*} H)^{2r}}
		\le 
		2^{2r}\v^{*}_\zs{r}\,\left( 
		\frac{1}{H^{r}}
		+
		\frac{1}{H^{2(1-\delta^{*})r}}
		\right)\,.
		$$
		\noindent Tsking into account here that $0<\delta^{*}<1/2$ and that $H\ge 1$ we can estimate this probability as
		\begin{equation}
			\label{sec:bnd-nu}
			\sup_\zs{\theta\in\Theta}\,
			\P_\zs{\theta}\,
			\left(
			\upsilon_\zs{H}
			>\n^{*}_\zs{H}
			\right)
			\le\,\frac{2^{2r+1}\v^{*}_\zs{r}}{H^{r}}
		\end{equation}
		
	\noindent and, therefore, using this in \eqref{ineq:prob-tau3} we get
	\begin{align}\label{eq:tau3}
		\sup_\zs{\theta\in\Theta}\,
		\P_\zs{\theta} (\tau_\zs{H}>T)\le \frac{T^m \r^{2m} Z_\zs{m}}{(\mu_\zs{*}T-H)^{2m}}+
		\frac{2^{2r+1}\v^{*}_\zs{r}}{H^{r}}\,.
	\end{align}
	\noindent  Now, from \eqref{eq:varalbe} and \eqref{eq:varalbe-bis} it follows that for any $m>1$ and $r>1$
	$$
	\sup_{\theta\in\Theta}\E_\zs{\theta}\ |	\wt{\theta}_\zs{H,T}-\theta|^2 \le
	\frac{(2\u_\zs{*}+\rho^{*}_\zs{H}) \sigma}{H}+\frac{T^m\,\theta_\zs{max} \r^{2m} Z_\zs{m}}{(\mu_\zs{*}T -H)^{2m}}  		
	+
	\frac{2^{2r+1}\v^{*}_\zs{r}{\theta_\zs{max}}}{H^{r}}\,.		
	$$
	\noindent Taking here {$r=2m$} we obtain the bound \eqref{sec:MultPr--fxd-ac-00}.
\end{proof}

\medskip   

\noindent  Now we consider the main term in the mean square accuracy in \eqref{sec:MultPr--fxd-ac-00} setting
\begin{equation}
	\label{sec:2-param--1-acc}    
	\e_\zs{m}(H,T)= \frac{2\u_\zs{*} \sigma}{H}+\frac{T^m\,\theta_\zs{max} \r^{2m} Z_\zs{m}}{(\mu_\zs{*}T -H)^{2m}}\,.
\end{equation}

\noindent Now  similarly to the definition
\eqref{sec;Opt-H} we choose an optimal value for the parameter $H$ 
to minimise this function, i.e.
\begin{equation}
	\label{sec;Opt-H-ab-1}
	H^{*}_\zs{T}=\argmin_\zs{0<H< \mu_\zs{*} T} \e_\zs{m}(H,T)\,.
\end{equation}
\noindent  Using this parameter in  \eqref{eq:TrunEst-ab}, we obtain the following sequential estimation 
procedure $\big(\tau^{*}_\zs{T},\theta^{*}_\zs{T} \big)$, in which
\begin{equation}
	\label{sec;Opt-H-Seq-Pr-abb}
	\tau^{*}_\zs{T}=\wt{\tau}_\zs{H^{*}_\zs{T},T} 
	\quad\mbox{and}\quad
	\theta^{*}_\zs{T}=
	\wt{\theta}_\zs{H^{*}_\zs{T},T}
	\,.
\end{equation}

\noindent Now similarly to Corollary
\ref{cor:b--aa} we can show the following result

\begin{corollary}\label{cor:b--abb-1}
	Assume that for some $0<\delta<1/2$ the parameter $\r$ in the definitions \eqref{muu-03-01}
	such that
	\begin{equation}
		\label{sec:Cond-1-2pr}
		\r=\mathrm{O}(T^{\delta})
		\quad\mbox{as}\quad
		T\to\infty\,.
	\end{equation}
	\noindent Then, for  any $m>(1-2\delta )^{-1}$ 
	the optimal truncated procedure  
	\eqref{sec;Opt-H-Seq-Pr-a}
	posses the following asymptotic 
	properties:.
	
	\begin{enumerate}
		
		\item 
		the optimal parameter
		\eqref{sec;Opt-H-ab-1}  for $T\to\infty$
		is represented as
		\begin{equation}\label{H1-2prms}
			H^{*}_\zs{T}=\mu_\zs{*} T-
			\r^{\frac{2m}{2m+1}}\,
			\left(m\mu^{2}_\zs{*}Z_\zs{m}/(\u_\zs{*}\sigma))\right)^{\frac{1}{2m+1}}( T)^{\frac{2+m}{2m+1}}(1+\mathrm{o}(1))
			=\bar{\mu}_\zs{*} T+\oo(T)\,;
		\end{equation} 
		\item
		\noindent 
		the corresponding optimal estimation accuracy for $T\to\infty$
		has the following form 
		\begin{equation}\label{accuracy-H2--a-b}
			\sup_{\theta\in\Theta}\E_\zs{\theta}\big( \theta^{*}_\zs{T}-\theta\big)^2\le
			\e_\zs{m}(H^{*}_\zs{T},T)
			+\mathrm{o}\left(\frac{1}{T^{2m}}\right)
			=
			\frac{2\u_\zs{*}\sigma}{\bar{\mu}_\zs{*} T}+\mathrm{o}\left(\frac{1}{T}\right) \,,
		\end{equation}
	\end{enumerate}
	\noindent where 
	$
	\bar{\mu}_\zs{*}
	=
	\min_\zs{(a,b)\in\Theta}\,
	\tr \, F$.
\end{corollary}

\subsection{Optimality properties}\label{sec:AsOpt}

\noindent 
Now we study the optimality properties for the procedure
\eqref{sec;Opt-H-Seq-Pr-abb}.  First we study the stoping moment 
\eqref{sec:MultPr-1-00}.

\begin{proposition}\label{Pr.sec:Mult-time-1tt}
	For any compact set $\Theta\subset (\sigma/2,+\infty)\, \times\, (0,+\infty)$ and for any $r>0$
	\begin{equation}
		\label{sec:mult-tmi-1-unif}
		\lim_\zs{H\to\infty}
		\sup\limits_\zs{\theta\in \Theta}\E_\zs{\theta}
		\left|
		\frac{\t_\zs{H}}{H}-
		\frac{1}{\tr\,F}
		\right|^{r}
		=0
		\,,
	\end{equation}
	where the matrix $F$ is defined in \eqref{sec:Mlt.cond-1}.
\end{proposition}
\begin{proof}
	First we show that for any $\varepsilon>0$
	\begin{equation}
		\label{sec:Opt-tt-h--02}
		\lim_\zs{H\to\infty}
		\sup_\zs{\theta\in\Theta}\,\P_\zs{\theta}\left(\big|\frac{\t_\zs{{H}}}{H}-\frac{1}{\tr\, F}\big|>\varepsilon\right)=0\,.
	\end{equation}
	\noindent To do this note that  for $0<\varepsilon<\min_\zs{\theta\in\Theta}(\tr\, F)^{-1}$ can be represented as
	$$
	\P_\zs{\theta}\left(\big|\frac{\t_\zs{{H}}}{H}-\frac{1}{\tr\, F}\big|>\varepsilon\right)=\P_\zs{\theta}\left(\t_\zs{{H}}>t_\zs{1}\right)+\P_\zs{\theta}\left(\t_\zs{{H}}<t_\zs{2}\right)\,,
	$$
	\noindent 
	where $t_\zs{1}=((\tr\, F)^{-1}+\varepsilon)H$ and $t_\zs{2}=((\tr\, F)^{-1}-\varepsilon)H$.   The first probability here can be estimated  as
	\begin{align*}
		\P_\zs{\theta}\left(\t_\zs{{H}}>t_\zs{1}\right)
		&=
		\P_\zs{\theta}\left( - \tr \,\left(G_\zs{t_\zs{1}} -  t_\zs{1} F \right) >t_\zs{1}\tr F-H\right)\\[2mm]
		&\le \P_\zs{\theta}\left( | \tr \,\left(G_\zs{t_\zs{1}} -  t_\zs{1} F \right)| >\varepsilon \tr F\, H\right)
		\le \P_\zs{\theta}\left( | \tr \,\left(G_\zs{t_\zs{1}} -  t_\zs{1} F \right)| >2\varepsilon_\zs{*} \,H\right)
	\end{align*}	
	where
	$
	\varepsilon_\zs{*}=\varepsilon \inf_\zs{\theta\in\Theta}\, \tr\, F/2$.
	\noindent
	Using here the definition of the matrix $F$ in \eqref{sec:Mlt.cond-1}, we get 
	$$
	\P_\zs{\theta}\left(\t_\zs{{H}}>t_\zs{1}\right)
	\le \P_\zs{\theta}\left(\big|\Upsilon_\zs{t_\zs{1}}\big| >\varepsilon_\zs{*} H \right)+
	\P_\zs{\theta}\left(| \D_\zs{t_\zs{1}} | > \varepsilon_\zs{*}H\right)\,,
	$$
	\noindent 
	where
	$\Upsilon_\zs{t}=\int_\zs{0}^{t}\left(X_\zs{u}^{-1}-f_\zs{1}\right)\d u$ and the deviation
	$\D_\zs{t}$ is defined in
	\eqref{ineq:prob-tau} for $\theta =b$.
	From \eqref{sec:Intr.CIR}, by Itô's formula, we have for any $t>0$,
	\begin{equation*}
		\label{ln-eq-100-09-Eq}
		\ln X_\zs{t}=\ln x +
		\frac{2 a-\sigma}{2}
		\,\Upsilon_\zs{t}
		+\sqrt{\sigma}\int_\zs{0}^{t}\,X^{-1/2}_\zs{u}
		\d W_\zs{u}
	\end{equation*}
	\noindent and, therefore,
	\begin{equation*}
		\label{UppBnd-LoGggn-1-0}
		\Upsilon_\zs{t}
		=
		\frac{2(\ln{X_\zs{t}}-\ln{x})}{2 a-\sigma}
		-
		\frac{2\sqrt{\sigma}}{2 a-\sigma}\int_\zs{0}^{t}\,X^{-1/2}_\zs{s}\d W_\zs{s}\,.
	\end{equation*}
	
	\noindent
	Using here the bound
	\eqref{MomentCIR-2}we get  that  for any $\theta\in\Theta$
	$$
	\E_\zs{\theta}\,\Upsilon_\zs{t}^{2}\,\le 
	\frac{12}{(2 a-\sigma)^{2}}
	\left( (\ln{x})^{2}+\E_\zs{\theta}(\ln{X_\zs{t}})^{2}
	+ \int_\zs{0}^{t} \E_\zs{\theta}X^{-1}_\zs{u}\d u
	\right)\,.
	$$
	\noindent Now, taking into account here that for any $\epsilon>0$
	\begin{equation*}
		\label{Upp-Lgn--12980}
		\sup_\zs{x>0}\frac{|\ln x|}{x^{\epsilon}+x^{-\epsilon}}<\infty\,,
	\end{equation*}
	\noindent 
	one can conclude that
	\begin{equation}
		\label{eq:Upsilon-NEW-1}
		\Upsilon^{*}=\sup_\zs{t\ge 1}\,
		\sup_\zs{\theta\in\Theta}\,
		\frac{\E_\zs{\theta}\,\Upsilon_\zs{t}^{2}}{t}\,
		<\infty\,.
	\end{equation}
	\noindent From here and \eqref{sec:CoIn-D-T-2-1} one can deduce that for any $H>1$ for which $t_\zs{1}\ge 1$, i.e. for $H\ge 1+ \max_\zs{\theta\in\Theta} \tr F$
	\begin{equation}
		\label{sec:--Up__Bnd--1}
		\sup_\zs{\theta\in\Theta}\,
		\P_\zs{\theta}\left(\t_\zs{{H}}>t_\zs{1}\right)
		\le 
		\frac{(\Upsilon^{*}+\D^{*}_\zs{1})\max_\zs{\theta\in\Theta}\,t_\zs{1}}{\varepsilon^{2}_\zs{*} H^{2}}
		\le
		\frac{(\Upsilon^{*}+\D^{*}_\zs{1})\big(f^{*}+\varepsilon\big)}{\varepsilon^{2}_\zs{*} H}\,,
	\end{equation}
	\noindent where $f^{*}=\max_\zs{\theta\in\Theta}\,(\tr\, F)^{-1}$. Therefore,
	$$
	\lim_\zs{H\to\infty}\,
	\sup_\zs{\theta\in\Theta}\,
	\P_\zs{\theta}\left(\t_\zs{{H}}>t_\zs{1}\right)=0\,.
	$$
	\noindent Moreover,  for any $\theta\in\Theta$
	\begin{align*}
		\P_\zs{\theta}\left(\t_\zs{{H}}<t_\zs{2}\right)
		&=
		\P_\zs{\theta}\left(  \tr \,\left(G_\zs{t_\zs{2}} -  t_\zs{2} F \right) >H-t_\zs{2}\tr F \right)\\[2mm]
		&= \P_\zs{\theta}\left(  \tr \,\left(G_\zs{t_\zs{2}} -  t_\zs{2} F \right) >\varepsilon \tr F\, H\right)
		\le \P_\zs{\theta}\left( | \tr \,\left(G_\zs{t_\zs{2}} -  t_\zs{2} F \right)| >2\varepsilon_\zs{*} \,H\right)\,.
	\end{align*}	
	\noindent Therefore,
	$$
	\P_\zs{\theta}\left(\t_\zs{{H}}<t_\zs{2}\right)
	\le
	\P_\zs{\theta}\left(\big|\Upsilon_\zs{t_\zs{2}}\big| >\varepsilon_\zs{*} H \right)+
	\P_\zs{\theta}\left(| \D_\zs{t_\zs{2}} | > \varepsilon_\zs{*}H\right)
	$$
	\noindent  and similarly to
	\eqref{sec:--Up__Bnd--1}
	we can conclude the following limit equality
	$$
	\lim_\zs{H\to\infty}
	\sup_\zs{\theta\in\Theta}\,
	\P_\zs{\theta}\left(\t_\zs{{H}}<t_\zs{2}\right)=0\,,
	$$
	\noindent which implies  \eqref{sec:Opt-tt-h--02}.  Now we need to show that for any $r>0$
	\begin{equation}\label{Upp-st-Tim--101}
		\overline{\lim}_\zs{H\to\infty}
		\sup_\zs{\theta\in\Theta}\,
		\frac{\E_\zs{\theta}\t^r_\zs{H}}{H^{r}}<\infty\,.
	\end{equation}
	\noindent To this end setting $\gamma_\zs{*}=\b_\zs{max}/\a_\zs{min}$ this moment can be estimated as
	\begin{align*}
		\E_\zs{\theta}\,\t^r_\zs{H}
		&= r\int_\zs{0}^\infty t^{r-1}\P_\zs{\theta}(\t_\zs{H}>t)\,\d t
		= r\int_\zs{0}^\infty t^{r-1}\P_\zs{\theta}(\tr \,G_\zs{t}\,<H)\d t\\[2mm]
		&\le 2^{r} \gamma^{r}_\zs{*}H^{r}  + r\int_\zs{2 \gamma_\zs{*}H}^\infty t^{r-1}
		\P_\zs{\theta}\left(\int_\zs{0}^{t}\,X_\zs{s}\d s+\int_\zs{0}^{t}\,X^{-1}_\zs{s} \d s<H\right)\d t\\[2mm]	
		\le\,&  2^{r} \gamma^{r}_\zs{*}H^{r} 
		+ r\int_\zs{2 \gamma_\zs{*}H}^\infty\, t^{r-1}
		\P_\zs{\theta}\left( |\D_\zs{t}|>f_\zs{2} t-H\,\right)\d t\,.
	\end{align*}
	\noindent Taking into account here that $f_\zs{2}\ge 1/\gamma_\zs{*}$
	and using the bound
	\eqref{Th.sec:CoIn-00}	we obtain that sufficiently large $H$
	\begin{align*}
		\E_\zs{\theta}\,\t^r_\zs{H}&\le 
		2^{r} \gamma^{r}_\zs{*}H^{r} 
		+ r \gamma^{2m}_\zs{*} \D^{*}_\zs{m}
		\int_\zs{2 \gamma_\zs{*}H}^\infty\, 
		\frac{t^{r-1+m}}{(t-\gamma_\zs{*}H)^{2m}}\,		
		\d t\\[2mm]
		&\le
		2^{r} \gamma^{r}_\zs{*}H^{r} 
		+ r 2^{r+m-2}\gamma^{2m}_\zs{*} \D^{*}_\zs{m}\,\left(
		\int_\zs{\gamma_\zs{*}H}^\infty\, 
		\frac{1}{x^{m-r+1}}\,\d x		
		+
		\frac{1}{(2m-1)(\gamma_\zs{*}H)^{m-r}}
		\right)\,.		
	\end{align*}
	\noindent Choosing here $m>r$ we obtain the property  \eqref{Upp-st-Tim--101} which together with
	the equality
	\eqref{sec:Opt-tt-h--02} implies  \eqref{sec:mult-tmi-1-unif}.
\end{proof}

\begin{proposition}\label{Pr.sec:Mult-time-1-draft}
	For any compact set $\Theta\subset (\sigma/2,+\infty)\, \times\, (0,+\infty)$ and for any $r>0$, the  duration time in the sequential 
	procedure 
	\eqref{sec;Opt-H-Seq-Pr-abb}
	defined through the sequence \eqref{sec:seq-kappa}-\eqref {sec:seq-kappa--0} satisfies
	the following limit property
	\begin{equation}
		\label{sec:mult-tmi-1-unif-**}
		\lim_\zs{H\to\infty}
		\sup\limits_\zs{\theta\in \Theta}
		\E_\zs{\theta}\,
		\left|
		\frac{\tau^{*}_\zs{T}}{T}-
		\frac{\bar{\mu}_\zs{*}}{\tr\, F}
		\right|^{r}
		=0
		\,.
	\end{equation}
	\noindent the coefficient $\bar{\mu}_\zs{*}$ is defined in \eqref{accuracy-H2--a-b}.
\end{proposition}
\begin{proof}
	Note here that
	$$
	\E_\zs{\theta}\,
	\left|
	\frac{\tau^{*}_\zs{T}}{T}-
	\frac{\bar{\mu}_\zs{*}}{\tr\, F}
	\right|^{r}\le 
	\E_\zs{\theta}\,
	\left|
	\frac{\tau_\zs{H^{*}_\zs{T}}}{T}-
	\frac{\bar{\mu}_\zs{*}}{\tr\, F}
	\right|^{r}\,\Chi_\zs{\{\tau_\zs{H^{*}_\zs{T}}\le T \}}+
	\P_\zs{\theta}\left( \tau_\zs{H^{*}_\zs{T}}> T\right)\,.		
	$$
	\noindent The first expectation can be estimated as
	\begin{align*}
	\E_\zs{\theta}\,
	\left|
	\frac{\tau_\zs{H^{*}_\zs{T}}}{T}-
	\frac{\bar{\mu}_\zs{*}}{\tr\, F}
	\right|^{r}\,\Chi_\zs{\{\tau_\zs{H^{*}_\zs{T}}\le T \}}& \le
	\E_\zs{\theta}\,
	\left|
	\frac{\t_\zs{H^{*}_\zs{T}}}{T}-
	\frac{\bar{\mu}_\zs{*}}{\tr\, F}
	\right|^{r}\,\Chi_\zs{\{\tau_\zs{H^{*}_\zs{T}}\le T \}\cap \{\upsilon_\zs{H^{*}_\zs{T}}
		\le \n^{*}_\zs{H^{*}_\zs{T}}\}} \\[2mm]
	&+\P_\zs{\theta}\,
	\left(
	\upsilon_\zs{H^{*}_\zs{T}}
	>\n^{*}_\zs{H^{*}_\zs{T}}
	\right)		
	\end{align*}
	\noindent and, therefore,
	
	$$
	\E_\zs{\theta}\,
	\left|
	\frac{\tau^{*}_\zs{T}}{T}-
	\frac{\bar{\mu}_\zs{*}}{\tr\, F}
	\right|^{r}\le 
	\E_\zs{\theta}\,
	\left|
	\frac{\t_\zs{H^{*}_\zs{T}}}{T}-
	\frac{\bar{\mu}_\zs{*}}{\tr\, F}
	\right|^{r}\, 
	+\P_\zs{\theta}\,\left(
	\upsilon_\zs{H^{*}_\zs{T}}
	>\n^{*}_\zs{H^{*}_\zs{T}}
	\right)		
	+
	\P_\zs{\theta}\left( \tau_\zs{H^{*}_\zs{T}}> T\right)\,.		
	$$
	
	\noindent Now, using  {the equality \eqref{H1-2prms}
	and the condition  \eqref{sec:Cond-1-2pr}}
	in the bound 
	\eqref{eq:tau3}
	we obtain that
	\begin{equation}
		\label{sec:St-Tim=Bnd-1}
		\lim_\zs{T\to\infty}\,T
		\sup_\zs{\theta\in\Theta}\,
		\P_\zs{\theta}\left( \tau_\zs{H^{*}_\zs{T}}> T\right)
		=0\,.
	\end{equation}
	\noindent 
	The bound   \eqref{sec:bnd-nu}
	and Proposition  \ref{Pr.sec:Mult-time-1tt}
	{combined with \eqref{H1-2prms}} imply the property \eqref{sec:mult-tmi-1-unif}.
\end{proof}

\medskip

\begin{theorem}
	\label{Th.sec:SeqEst-1-a-bb-0}
	The sequential procedure \eqref{sec;Opt-H-Seq-Pr-abb} is pointwise optimal, i.e for any parameter
	$\theta_\zs{0}\in (\sigma/2,+\infty)\, \times\, (0,+\infty)$
	\begin{equation}
		\label{sec:OPT-EstPrs-1-a-b}
		\lim_\zs{\gamma\to 0}\,
		\underline{\lim}_\zs{T\to\infty}
		\frac{
			\inf_\zs{
				(\tau,\wh{\theta}_\zs{\tau})
				\in
				\cH_\zs{T}(\theta_\zs{0},\gamma)}
			\sup_\zs{|\theta-\theta_\zs{0}|<\gamma}\,
			\E_\zs{\theta}\big| F^{1/2}(\theta_\zs{0}) (\wh{\theta}_\zs{\tau}-\theta)\big|^2}{\sup_\zs{|\theta-\theta_\zs{0}|<\gamma}\,
			\E_\zs{\theta}\big|F^{1/2}(\theta_\zs{0})(\theta^{*}_\zs{T}-\theta)\big|^2}
		=1\,,
	\end{equation}			
	\noindent wher the class $\cH_\zs{T}(\theta_\zs{0},\gamma)$  and  the matrix $F=F(\theta)$  are defined in \eqref{sec:AsOpt-1} and \eqref{sec:Mlt.cond-1}
	respectively.
\end{theorem}

\begin{proof}
	Note that according to Theorem 9 from  \cite{BenAlayaetal2025}
	for  any  $\theta=(a,b)^{\top}\in (\sigma/2,+\infty)\, \times\, (0,+\infty)$  the process \eqref{sec:Intr.CIR}
	satisfies the LAN condition with the   Fisher information $I(\theta)=F(\theta)/\sigma$. 
	Therefore, in view of Proposition \ref{Pr.sec:AsOp.-1}  for $k=2$
	we obtain that for any
	$\theta\in (\sigma/2,+\infty)\, \times\, (0,+\infty)$
	and for any $\gamma>0$ for which 
	$$
	\big\{\theta\in\bbr^{2}\,:\, |\theta-\theta_\zs{0}|<\gamma \big\} \subset (\sigma/2,+\infty)\, \times\, (0,+\infty)
	$$
	\noindent the following lower bound holds true   
	\begin{equation}
		\label{sec:AsOpt-2-a--b-0}
		\underline{\lim}_\zs{T\to\infty}
		\inf_\zs{(\tau,\wh{\theta}_\zs{\tau})\in\cH_\zs{T}(\theta_\zs{0},\gamma)}\,T
		\sup_\zs{| \theta-\theta_\zs{0}|<\gamma}\E_\zs{\theta}\,|F^{1/2}(\wh{\theta}_\zs{\tau}-\theta)|^{2}
		\ge 2 \sigma\,.
	\end{equation}
	\noindent Moreover, note that
	\begin{equation}\label{eq:err-44}
		\E_\zs{\theta}\,\big|
		F^{1/2}(\theta)\,(\theta^{*}_\zs{T}
		-\theta)
		\big |^{2}
		\le\, \E_\zs{\theta}\,\big |
		F^{1/2}(\theta)\, (\ov{\theta}_\zs{H^{*}_\zs{T}}
		-\theta)
		\big|^{2}\, + 
		|F^{1/2}\theta|^2\, {	\P_\zs{\theta}}(\tau_\zs{H^{*}_\zs{T}}> T)\,.
	\end{equation}
	\noindent To study the first term we use  Theorem 3.3 from
	\cite{BenAlayaetal2025} according to which for any compact set $\Theta (\sigma/2,+\infty)\, \times\, (0,+\infty)$
	
	\begin{equation}
		\label{sec:Upp-bnd-acc--01}
		\overline{\lim}_\zs{T\to\infty}
		\sup\limits_\zs{\theta\in \Theta}
		\m^{*}_\zs{T}(\theta)\,
		\E_\zs{\theta}
		\left| F^{1/2}(\theta)(\ov{\theta}_\zs{H^{*}_\zs{T}}
		-
		\theta)
		\right|^{2}
		\le 
		2\sigma\,,
	\end{equation}
	\noindent where $\m^{*}_\zs{T}(\theta)=\E_\zs{\theta} \tau_\zs{H^{*}_\zs{T}}$.
	Moreover, using here
	Theorem 3.2 from
	\cite{BenAlayaetal2025} one can conclude that
	
	\begin{equation*}
		\label{sec:Upp-bnd-acc--01-mod}
		\overline{\lim}_\zs{T\to\infty}
		\sup\limits_\zs{\theta\in \Theta}
		\frac{H^{*}_\zs{T}}{\tr\,F(\theta)}
		\,
		\E_\zs{\theta}
		\left| F^{1/2}(\theta)
		(\ov{\theta}_\zs{H^{*}_\zs{T}}
		-
		\theta)
		\right|^{2}
		\le 
		2\sigma
	\end{equation*}
	\noindent and, therefore, in view of  the representation {\eqref{H1-2prms}}
	\begin{equation*}
		\label{sec:UBn-02-mod-}
		\overline{\lim}_\zs{T\to\infty}T\,
		\sup\limits_\zs{\theta\in \Theta}
		\frac{\bar{\mu}_\zs{*}}{\tr\,F(\theta)}
		\,
		\E_\zs{\theta}
		\left| F^{1/2}(\theta)
		(\bar{\theta}_\zs{H^{*}_\zs{T}}
		-
		\theta)
		\right|^{2}
		\le 
		2\sigma\,.
	\end{equation*}
	\noindent We choose here $\Theta=\{\theta\in\bbr^{2}\,:|\theta-\theta_\zs{0}|<\gamma\}$
	and note that
	$$
	\left| F^{1/2}(\theta)
	(\bar{\theta}_\zs{H^{*}_\zs{T}}
	-
	\theta)
	\right|^{2}
	\ge
	\lambda_\zs{min}\Big(F^{-1/2}(\theta_\zs{0})F(\theta)F^{-1/2}(\theta_\zs{0})\Big)
	\left| F^{1/2}(\theta_\zs{0})
	(\bar{\theta}_\zs{H^{*}_\zs{T}}
	-
	\theta)
	\right|^{2}\,,
	$$
	\noindent where $\lambda_\zs{min}(G)$ is the minimal eigenvalue of the matrix $G$. Therefore, taking into account that
	$F(\theta)\to F(\theta_\zs{0})$ as $\theta\to\theta_\zs{0}$ and using
	the property \eqref{sec:St-Tim=Bnd-1} in {\eqref{eq:err-44}} 
	we obtain that
	\begin{equation*}
		\label{sec:UBn-02-mod-}
		\lim_\zs{\gamma\to 0}
		\overline{\lim}_\zs{T\to\infty}T\,
		\sup\limits_\zs{|\theta-\theta_\zs{0}|<\gamma}
		\,
		\E_\zs{\theta}
		\left| F^{1/2}(\theta_\zs{0})
		(\theta^{*}_\zs{T}
		-
		\theta)
		\right|^{2}
		\le 
		2\sigma\,,
	\end{equation*}
	\noindent which   together with the lower bound \eqref{sec:AsOpt-2-a--b-0} implies the property \eqref{sec:OPT-EstPrs-1-a-b}.
\end{proof}

\noindent  Now we study the minimax properties for the sequential procedure \eqref{sec;Opt-H-Seq-Pr-abb}.

\begin{theorem}\label{Th.sec:AsOp.-2-Mult-Ver}
	For any compact set $\Theta\subset (\sigma/2\,,\,+\infty) \times (0\,,\,+\infty)$
	the
	sequential procedure  \eqref{sec;Opt-H-Seq-Pr-abb} 
	is asymptotically optimal in the minimax sense, i.e.
	\begin{equation}
		\label{sec:AsOpt-SeqPr-Minimax}
		\lim_\zs{T\to\infty}\,
		\frac{
			\inf_\zs{(\tau,\wh{\theta}_\zs{\tau})\in\Xi^{*}_\zs{T}}\,
			\sup_\zs{ \theta\in\Theta}\,
			\E_\zs{\theta}\,\big|\,\wt{F}^{1/2}\,(\wh{\theta}_\zs{\tau}-\theta)\big|^{2}
		}
		{\sup_\zs{ \theta\in\Theta}\,
			\E_\zs{\theta}\,\big|\,\wt{F}^{1/2}\,(\theta^{*}_\zs{T}
			-\theta)\big|^{2}}
		=1\,,
	\end{equation}
	where the class $\Xi^{*}_\zs{T}$ is defined in \eqref{sec;AssOpt-GlCls-1bis} and the matrix $\wt{F}=\wt{F}(\theta)=F/ \tr\,F$. 
\end{theorem}

\begin{proof}
	Note that  the {property \eqref{sec:mult-tmi-1-unif-**}} implies the condition  $\C_\zs{1})$. Therefore, using 
	Theorem
	\ref{Th.sec:AsOp.-1bis} for $k=2$
	we obtain that
	\begin{align}\nonumber
		\underline{\lim}_\zs{T\to\infty}\,&T\,\ov{\mu}_\zs{*}\,
		\inf_\zs{(\tau, \wh{\theta}_\zs{\tau})\in\Xi^{*}_\zs{T}}\,
		\sup_\zs{ \theta\in\Theta}\,\E_\zs{\theta}\,|\,\wt{F}^{1/2} (\wh{\theta}_\zs{\tau}-\theta)|^2
\\[2mm]\label{sec:33--01-abf-01}		
	&	=
		\underline{\lim}_\zs{T\to\infty}
		\inf_\zs{(\tau, \wh{\theta}_\zs{\tau})\in\Xi^{*}_\zs{T}}\,\m^{*}_\zs{T}(\theta)\,
		\sup_\zs{ \theta\in\Theta}\,\E_\zs{\theta}\,|\,F^{1/2} (\wh{\theta}_\zs{\tau}-\theta)|^2		
		\ge 2\,\sigma\,.
	\end{align}
	\noindent Moreover, taking into account in  \eqref{eq:err-44} that the expectation $\m^{*}_\zs{T}(\theta)=\E_\zs{\theta} \tau^{*}_\zs{T}\le \E_\zs{\theta}\tau_\zs{H^{*}_\zs{T}}$ and  
	$\m^{*}_\zs{T}(\theta)=\E_\zs{\theta}\tau^{*}_\zs{T}\le T$,  we get that
	\begin{equation*}\label{eq:err}
		\m^{*}_\zs{T}(\theta)
		\E_\zs{\theta}\,\big|
		F^{1/2}(\theta)\,(\theta^{*}_\zs{T}
		-\theta)
		\big |^{2}
		\le\, \E_\zs{\theta}\tau_\zs{H^{*}_\zs{T}}
		\E_\zs{\theta}\,\big |
		F^{1/2}(\theta)\, (\ov{\theta}_\zs{H^{*}_\zs{T}}
		-\theta)
		\big|^{2}\ +T\,|F^{1/2}\theta|^2\, {	\P_\zs{\theta}}(\tau_\zs{H^{*}_\zs{T}}> T)\,.   
	\end{equation*}
	
	\noindent Now, Theorem 3.3 from
	\cite{BenAlayaetal2025} and  the property \eqref{sec:St-Tim=Bnd-1}
	yield the following upper bound
	$$
	\ov{\lim}_\zs{T\to\infty}\,T\,\ov{\mu}_\zs{*}\,
	\sup_\zs{ \theta\in\Theta}\,\E_\zs{\theta}\,|\,\wt{F}^{1/2} (\wh{\theta}_\zs{\tau}-\theta)|^2=
	\ov{\lim}_\zs{T\to\infty}\,\m^{*}_\zs{T}(\theta)\,
	\sup_\zs{ \theta\in\Theta}\,\E_\zs{\theta}\,|\,F^{1/2} (\theta^{*}_\zs{T}-\theta)|^2		
	\le 2\,\sigma\,,
	$$
	\noindent which together with  the lower bound \eqref{sec:33--01-abf-01} implies the property  \eqref{sec:AsOpt-SeqPr-Minimax}.
\end{proof}

\begin{remark}
	\label{Re.sec:opt-a-b--01}
	It should be noted that the property  \eqref{sec:mult-tmi-1-unif-**}
 imply that	
for $T\to\infty$ the mean observations duration $\E_\zs{\theta} \tau^{*}_\zs{T}-T\to-\infty$ for 
$\bar{\mu}_\zs{*} \neq \tr\, F$. 
Therefore, as is noted in Remark
  \eqref{Re.sec:opt--1} the procedure \eqref{sec;Opt-H-Seq-Pr-abb} has the same property as  the procedures  \eqref{sec;Opt-H-Seq-Pr} and  \eqref{sec;Opt-H-Seq-Pr-a}, 
  i.e. to provide the optimality properties it essentially reduces the duration of observations 
  compared with
   the non-sequential estimation based on the fixed observations  duration $T$.
\end{remark}

\section{Concentration inequalities for  the CIR models.}\label{sec:CoIn}

In this section we study the properties of the deviation in the ergodic theorem for the process \eqref{sec:Intr.CIR}.
First we study the deviation problem fir this process 
with the fixed parameter $a$ and $b=\theta$.  First we study the deviation \eqref{ineq:prob-tau}.

\begin{theorem}
	\label{Th.sec:CoIn-00}
	For any compact set {$\Theta\subset  (0,+\infty)$} and 
	for any $m\ge1$ 
	\begin{equation}
		\label{sec:CoIn-D-T-2-1}
		\D^{*}_\zs{m}=
		\sup_\zs{T\ge 1}
		\sup_\zs{\theta\in\Theta}
		\frac{
			\E_\zs{\theta}\,
			\D_\zs{T} ^{2m}}{T^{m}}
		\le 
		{\frac{\L_\zs{m}}{\b^{2m}_\zs{min}}}\,,
	\end{equation}
	\noindent {where the constants  $\b_\zs{min}$ and $\L_\zs{m}$ 
	are defined in  \eqref{sec:Para-1}
	and
	\eqref{sec:LL-m-1} respectively}.
\end{theorem}
\begin{proof} 
	First note that   from \eqref{sec:Intr.CIR} that for any $T\ge 1$ we can write the term $\D_\zs{T}$ as
	\begin{equation}\label{Rep-D-TT--01210}
		\D_\zs{T}=\frac{X_\zs{0}-X_\zs{T}}{\theta}+\frac{\sqrt{\sigma}}{\theta}\int_\zs{0}^T\sqrt{X_\zs{s}}\d W_\zs{s}\,.
	\end{equation}
	\noindent 
	Now,  by   the upper bound \eqref{MomentCIR-2}, we obtain that for any  $m>1$ and $\theta\in\Theta$,
	\begin{align}\nonumber
		\E_\zs{\theta}\,\D_\zs{T}^{2m}
		&	\le 3^{2m-1}\,
		\left(\frac{X_\zs{0}^{2m}+\E_\zs{\theta}X^{2m}_\zs{T}}{\theta^{2m}}+  \left(\frac{\sqrt{\sigma}}{\theta}\right)^{2m}\,
		\,
		\E_\zs{\theta}\left(\int_\zs{0}^T\sqrt{X_\zs{s}}\d W_\zs{s}\right)^{2m}\right)\\[2mm] \label{eq:Delta-1}
		&		\le 
		\frac{3^{2m-1}}{\b^{2m}_\zs{min}}\,
		\left(2\x_\zs{2m}	
		+\sigma^{m}\,
		\E_\zs{\theta}\left(\int_\zs{0}^T\sqrt{X_\zs{s}}\d W_\zs{s}\right)^{2m}\right)
		\,.
	\end{align}
	\noindent  
	Now,  using here the upper bounds \eqref{MomentCIR-2} and \eqref{sec:uppBnd-SI-1} we can get that
	$$
	\sup_\zs{\theta\in\Theta}\,
	\E_\zs{\theta}\left(\int_\zs{0}^T\sqrt{X_\zs{s}}\d W_\zs{s}\right)^{2m}
	\le  
	\x_\zs{m}\,(m(2m-1))^{m}\,T^{m}\,
	\,.
	$$
	\noindent The use of this bound in \eqref{eq:Delta-1} 
	for $T\ge 1$
	implies the inequality \eqref{sec:CoIn-D-T-2-1}.
\end{proof} 

\noindent 
Now to study the deviations of the form
\eqref{Deviation-2} 
for any continuous and bounded $\bbr_\zs{+}\to\bbr$ function $\phi$ we set the general form deviation as
\begin{equation} 
	\label{sec:CoIn-1}
	\Delta_\zs{T}(\phi)=
	\int^{T}_\zs{0}\left(
	\phi(X_\zs{t})
	-\mu_\zs{\theta}(\phi)
	\right)\d t\,,
\end{equation}
where $\mu_\zs{\theta}(\phi)=\int_\zs{\bbr_\zs{+}}\,\phi(z)\,\q_\zs{\theta}(z) \d z$ and the density 
$\q_\zs{\theta}$
is defined
\eqref{ergdenssityCIR}
for  $\theta=(a,b)^\top$.

\begin{theorem}
	\label{Th.sec:CoIn-2}
	For any compact set {$\Theta\subset (\sigma/2,+\infty)$}, for any $m\ge1$ and any continuous and bounded $\bbr_\zs{+}\to\bbr$ function $\phi$
	\begin{equation}
		\label{sec:CoIn-2-1}
		\Delta^{*}_\zs{T}=
		\sup_\zs{T\ge 1}\,
		\sup_\zs{\theta\in\Theta}
		\frac{
			\E_\zs{\theta}\,
			|\Delta_\zs{T}(\phi) |^{2m}}{T^{m}}
		< 
		\frac{\phi_\zs{*}^{2m}}{\sigma^{2m}}
		\L_\zs{m}
		\sup_\zs{\theta\in\Theta} 
		\left(\frac{4\,e^{\beta }}{\alpha}+\frac{2^{\alpha}\Gamma(\alpha)}{ \beta^{\alpha}}
		+\,
		\frac{2^{\alpha}\,}{ \beta}\right)^{2m}
		\,,
	\end{equation}
	\noindent where $\L_\zs{m}$ is defined in \eqref{sec:LL-m-1} {and $\phi_\zs{*}=\sup_\zs{u\in\bbr_\zs{+}}|\phi(u)|$}.
\end{theorem}
\begin{proof} 
	We use the method proposed in \cite{GaltchoukPergamenshchikov2007}.  According to this method we need to find a bounded solution $y(x)$ of the differential equation
	\begin{equation} 
		\label{sec:CoIn-1-diif-1}
		\frac{\sigma}{2} x\, \dot{y}(x) + (a-b x) y(x)=\wt{\phi}(x)
		\quad\mbox{and}\quad
		\wt{\phi}(x)=\phi(x)-\mu_\zs{\theta}(\phi)\,.
	\end{equation}
	
	\noindent One can check directlyb that in this case such solution can be represented as
	\begin{equation}
		\label{sec:CoIn-2-33}
		y(x)=-
		\frac{2}{\sigma\,x^{\alpha}}
		\int^{+\infty}_\zs{x} \wt{\phi}(u)\,
		u^{\alpha-1}
		\,e^{-\beta (u-x)}\,\d u\,,\quad
		\alpha=\frac{2 a}{\sigma}
		\quad\mbox{and}\quad
		\beta=\frac{2 b}{\sigma}
		\,.
	\end{equation} 
	\noindent  Note that the function
	$\wt{\phi}(u)$ is bounded, i.e. 
	$$
	\sup_\zs{u\in\bbr_\zs{+}}\,|\wt{\phi}(u)|
	\le 2\phi_\zs{*}
	\quad\mbox{and}\quad
	\phi_\zs{*}=\sup_\zs{u\in\bbr_\zs{+}}|\phi(u)|\,.
	$$
\noindent Using this we obtain that for all $x\ge 1$
	\begin{align*}
		|y(x)| \le 
		\,
		\frac{4\phi_\zs{*}}{\sigma x^{\alpha}}
		\int^{+\infty}_\zs{x}\,
		u^{\alpha-1}
		\,e^{-\beta (u-x)}\,\d u
		&\le
		\,
		\frac{2^{\alpha}\phi_\zs{*}}{\sigma x^{\alpha}}
		\int^{+\infty}_\zs{0}\,
		z^{\alpha-1}
		\,e^{-\beta z}\,\d z
		+\,
		\frac{2^{\alpha}\,\phi_\zs{*}}{\sigma x}
		\int^{+\infty}_\zs{0}
		\,e^{-\beta z}\,\d z
		\,\\[2mm]
		&\le\,\frac{2^{\alpha}\phi_\zs{*}\Gamma(\alpha)}{\sigma x^{\alpha}\beta^{\alpha}}
		+\,
		\frac{2^{\alpha}\,\phi_\zs{*}}{\sigma x\beta}\,,
	\end{align*}
	\noindent  i.e.
	$$\sup_\zs{x\ge 1}|y(x)|\le \frac{2^{\alpha}\phi_\zs{*}\Gamma(\alpha)}{\sigma \beta^{\alpha}}
	+\,
	\frac{2^{\alpha}\,\phi_\zs{*}}{\sigma \beta}\,.
$$	
\noindent In the case, when $0<x<1$ taking into account, that 
	$$
	\int^{+\infty}_\zs{0} \wt{\phi}(u)\,
	u^{\alpha-1}
	\,e^{-\beta u}\,\d u
	=
	\int^{+\infty}_\zs{0} \phi(u)\,
	u^{\alpha-1}
	\,e^{-\beta u}\,\d u
	-
	\mu_\zs{\theta}(\phi)
	\int^{+\infty}_\zs{0} 
	u^{\alpha-1}
	\,e^{-\beta u}\,\d u
	=0
	\,,
	$$
	\noindent we can rewrite the solution \eqref{sec:CoIn-2-33} as
	$$
	y(x)=
	\frac{2e^{\beta x}}{\sigma\,x^{\alpha}}
	\int^{x}_\zs{0} \wt{\phi}(u)\,
	u^{\alpha-1}
	\,e^{-\beta u}\,\d u\,.
	$$
	\noindent So, for $0<x<1$
	$$
	|y(x)|
	\le
	\frac{4\phi_\zs{*}\,e^{\beta }}{\sigma\,x^{\alpha}}
	\int^{x}_\zs{0} 
	u^{\alpha-1}
	\,e^{-\beta u}\,\d u
	\le 
	\frac{4\phi_\zs{*}\,e^{\beta }}{\sigma\,\alpha}
	\,,
	$$
	\noindent and, therefore, for any $\theta\in\Theta$
	\begin{equation}
		\label{sec:UBND-y-*}
		y_\zs{*}=\sup_\zs{x\in\bbr_\zs{+}}\,|y(x)|\le \frac{\phi_\zs{*}}{\sigma}\left(\frac{4\,e^{\beta }}{\alpha}+\frac{2^{\alpha}\Gamma(\alpha)}{ \beta^{\alpha}}
		+\,
		\frac{2^{\alpha}\,}{ \beta}\right)\,.
	\end{equation}
	\noindent In view of the It\^o formula for the function $V(u)=\int^{u}_\zs{0} y(x)\d x$ and the equation \eqref{sec:CoIn-1-diif-1}, we obtain, that
	$$
	\Delta_\zs{T}(\phi)
	= V(X_\zs{T})
	-
	V(X_\zs{0})
	-
	\sqrt{\sigma}
	\int^{T}_\zs{0}\,y(X_\zs{t})\,\sqrt{X_\zs{t}} \,\d W_\zs{t}
	\,.
	$$
	
	\noindent
	Using now the moment inequality
	\eqref{MomentCIR-2}, we get, that for any $m\ge 1$
	$$
	\sup_\zs{T>0}
	\sup_\zs{\theta\in\Theta}\,
	\E_\zs{\theta} |V(X_\zs{T})|^{2m}
	\le 
	y^{2m}_\zs{*} 
	\sup_\zs{T>0}\,
	\sup_\zs{\theta\in\Theta}\,
	\E_\zs{\theta} X^{2m}_\zs{T}
	= y^{2m}_\zs{*} \x_\zs{2m}
	\,.
	$$
	\noindent Moreover,  through the bound
	\eqref{sec:uppBnd-SI-1}
	we obtain that for any $\theta \in \Theta$
\begin{align*}
	\E_\zs{\theta}
	\left(
	\int^{T}_\zs{0}\,y(X_\zs{t})\,\sqrt{X_\zs{t}} \,\d W_\zs{t}\right)^{2m}
	&\le
	\left(m(2m-1)\right)^{m}\,T^{m-1}\,y^{2m}_\zs{*}
	\int^{T}_\zs{0}\,\E_\zs{\theta} X^{m}_\zs{t}\,\d t\\[2mm]
&	\le 
	\left(m(2m-1)\right)^{m}\,T^{m}\,y^{2m}_\zs{*}\,
	\x_\zs{m}
	\,.
\end{align*}
	\noindent  
	Therefore, we can estimate the deviation \eqref{sec:CoIn-1} for $T\ge 1$
	as
	\begin{align*}
		\E_\zs{\theta}\,
		(\Delta_\zs{T}(\phi))^{2m}
		&\le 3^{2m-1}
		\left(
		\E_\zs{\theta}\,
		V^{2m}(X_\zs{T})
		+
		V^{2m}(X_\zs{0})
\right)\\[2mm]
		&		
		+3^{2m-1}\,
		\sigma^{m}\,
		\E_\zs{\theta}
		\left(
		\int^{T}_\zs{0}\,y(X_\zs{t})\,\sqrt{X_\zs{t}} \,\d W_\zs{t}\right)^{2m}\\[2mm]
		&\le 
		3^{2m-1}
		\left(
		2 \x_\zs{2m}
		+
		\sigma^{m}\,
		\left(m(2m-1)\right)^{m}\,
		\x_\zs{m}
		\right)\, y^{2m}_\zs{*}\,T^{m}\,.
	\end{align*}
	\noindent
	\noindent Using here the bound \eqref{sec:UBND-y-*}
	we obtain the inequality \eqref{sec:CoIn-2-1}.
\end{proof}

\section{Conclusion}\label{sec:conclusion}

\noindent 
In the conclusion, we emphasize that 

\begin{itemize}

	\item The truncated sequential estimation procedures are constructed and non-asymptotic mean square accuracies
	are obtained in
	\eqref{ineq:2},
	\eqref{ineq:1}
	and
	\eqref{sec:MultPr--fxd-ac-00}.  
	The properties for the mean  observations durations are studied in Propositions \ref{Pr.sec:SeqEst-1},  \ref{Pr.sec:SeqEst-1-a} and {\ref{Pr.sec:Mult-time-1-draft}}.

	\item  
	For the first time for continuous time statistical models,
	the optimality properties for the truncated sequential estimation
	procedures are established in the class of all possible sequential procedures with arbitrary bounded stopping times determining the duration of the observation.

\item	
To provide the optimality properties the proposed truncated sequential procedures use essentially fewer observations than classical non-sequential estimators based on the fixed non-random duration of observations 
(see Remarks \ref{Re.sec:opt--1} and \ref {Re.sec:opt-a-b--01}).

\end{itemize}

\bigskip

\section*{Acknowledgments}
The first two authors were benefited from the support of the ANR project "Efficient inference for large and high-frequency data" (ANR-21-CE40-0021). The last author was partially supported
by RSF Project No 24-11-00191 (National Research Tomsk State University). 

\section*{Conflicts of interest statement}
On behalf of all authors, the corresponding author states that there is no conflict of interest.

\bigskip
\setcounter{section}{0}
\renewcommand{\thesection}{\Alph{section}}

\medskip

\section{Appendix}\label{sec:Appendix}

\subsection{Local Asymptotic Normality property}\label{sec:App-0}
First we recall that a family of probability measures $(\P_\zs{\theta,T})_\zs{\theta\in\Theta, T>0}$ with $\Theta\subseteq \bbr^{k}$  is called to satisfy the Local Asymptotic Normality condition (LAN)
at a point $\theta_\zs{0}\in\Theta$
if there   is  a  scaling $k\times k$ matrix 
$\upsilon_\zs{T}$ going to zero as $T\to \infty$ such that for
any $u\in \bbr^{k}$ for which the point $\wt{\theta}=\theta_\zs{0}+\upsilon_\zs{T} u$ belongs to $\Theta$,
the Radon-Nikodym derivative has the following asymptotic representation 
\begin{equation}
	\label{sec:LAN-DEF}
	\ln
	\frac{\d \P_\zs{\wt{\theta},T}}{\d \P_\zs{\theta_\zs{0},T}}= u^{\top}\,\xi_\zs{T}
	-\frac{|u|^2}{2}+
	\r_\zs{T}(u)\,,
\end{equation}
where $|\cdot |$ is the euclidean norm in $\bbr^{k}$,
$$
\xi_\zs{T}  
\xrightarrow[T\to\infty]{\mathcal{L}(\P_\zs{\theta_\zs{0},T})}
\mathcal{N}(0,1_\zs{k})
\quad\mbox{and}\quad
\sup_\zs{\vert u\vert\le \u_\zs{*}}
\vert\r_\zs{T}(u)\vert
\xrightarrow[T\to\infty]{\P_\zs{\theta_\zs{0},T}}
0
\quad\mbox{for any}\quad
\u_\zs{*}>0\,.
$$
\noindent Here, $1_\zs{k}$ is the identity matrix of order $k$.

\noindent
To study the point-wise optimality properties
for sequential procedures
we will  use the lower bound obtained in
\cite[Proposition 6.1]{BenAlayaetal2025}
for the class
\eqref{sec:AsOpt-1}.

\begin{proposition}\label{Pr.sec:AsOp.-1}
	Assume that, LAN holds for  $\theta_\zs{0}$ from $\Theta$ with the scale matrix of the form $\upsilon_\zs{T}= (I(\theta_\zs{0}) T)^{-1/2}$, where $I(\theta_\zs{0})$ is some positive  definite  matrix. Then,
	for any  $\gamma>0$ for which $\{|\theta-\theta_\zs{0}|\le \gamma\}\subseteq \Theta$,   the following asymptotic lower bound holds true 
	\begin{equation}
		\label{sec:AsOpt-2}
		\underline{\lim}_\zs{T\to\infty}
		\inf_\zs{(\tau,\wh{\theta}_\zs{\tau})\in\cH_\zs{T}(\theta_\zs{0},\gamma)}\,
		\sup_\zs{| \theta-\theta_\zs{0}|<\gamma}\E_\zs{\theta}\,|\upsilon_\zs{T}^{-1}\,(\wh{\theta}_\zs{\tau}-\theta)|^2
		\ge k\,.
	\end{equation}
\end{proposition}
\noindent This lower bound will be used to study local optimality properties, i.e. for small vicinities of $\theta_\zs{0}$. 
To study optimality properties over some arbitrary compact set $\Theta\subset\bbr^{k}$ for the class of sequential procedures defined in 
\eqref{sec;AssOpt-GlCls-1bis}
we need the following conditions.

\noindent $\C_\zs{1})$ {\sl There exists $\theta_\zs{0}\in\Theta$, such that $\{|\theta-\theta_\zs{0}|<\gamma\}\subset\Theta$ for all sufficiently small $\gamma>0$ and}
\begin{equation}
	\label{sec;AssOpt-MeanTimeCnd--09--1}
	\lim_\zs{\theta\to\theta_\zs{0}}
	\overline{\lim}_\zs{T\to\infty}
	\left|
	\frac{\m^{*}_\zs{T}(\theta) }{\m^{*}_\zs{T}(\theta_\zs{0})}
	-
	1
	\right|
	=0\,,
\end{equation}
\noindent where $\m^{*}_\zs{T}(\theta) =\E_\zs{\theta}\tau^{*}_\zs{T}$.

\medskip

\noindent $\C_\zs{2})$ {\sl There exists $\theta_\zs{0}\in\Theta$ for which the  LAN condition holds true
	for the scale matrix of the form   $\upsilon_\zs{T}= I^{-1/2}(\theta_\zs{0}) T^{-1/2}$ in which $I(\theta)$ is positive defined and  continuous matrix for any $\theta$ from some neighborhood of the point $\theta_\zs{0}$ 
	in $\Theta$. }\\

\noindent In the sequel we will use the following lower for this class   obtained in \cite[Theorem 6.4]{BenAlayaetal2025}.

\begin{theorem}\label{Th.sec:AsOp.-1bis}
	Assume that  the conditions $\C_\zs{1})$ -- $\C_\zs{2})$ hold true for some $\theta_\zs{0}$ from $\Theta$. Then,
	\begin{equation}
		\label{sec:AsOpt-33--01}
		\underline{\lim}_\zs{T\to\infty}
		\inf_\zs{(\tau, \wh{\theta}_\zs{\tau})\in\Xi^{*}_\zs{T}}\,
		\sup_\zs{ \theta\in\Theta}\,\m^{*}_\zs{T} (\theta)\,\E_\zs{\theta}\,|\, I^{1/2}(\theta) (\wh{\theta}_\zs{\tau}-\theta)|^2
		\ge k\,.
	\end{equation}
\end{theorem}

\medskip

\subsection{Moment properties of the CIR process}\label{sec:App}
\noindent

\noindent Now we study the moment properties for the stable CIR processes
\begin{lemma}\label{Le.momentCIR-1}
	For any $q>- 2a/\sigma$ and compact set $\Theta\subset ]\sigma/2,+\infty[\, \times \, ]0,+\infty[$
	\begin{equation}\label{MomentCIR-2}
		\x_\zs{q}
		=
		\sup_\zs{t\ge 0}\,
		\sup_\zs{\theta\in\Theta}\,
		\E_\zs{\theta} X^{q}_\zs{t}
		<\infty\,.
	\end{equation}
\end{lemma}

\noindent The proof is given in Proposition 3 from \cite{BAK2013}.

\begin{remark}\label{Re.sec:App-121}
It should be noted that
for $q>0$
 the bound \eqref{MomentCIR-2} holds true
 for   any compact set $\Theta\subset ]0,+\infty[\, \times \, ]0,+\infty[$.
\end{remark}

\subsection{Properties of stochastic integrals}\label{sec:App-1--00}
\noindent 
Now we give the upper bound for the moments of the stochastic integrals.
\begin{lemma}\label{Le.sec:App--SI--1}
	Let $(f_\zs{t})_\zs{0\le t \le T}$ be adapted process such that for some $m>1$
	$$
	\E\,
	\int^{T}_\zs{0}\, f^{2m}_\zs{t}\d t\,
	<\infty\,.
	$$
	\noindent Then 
	\begin{equation}\label{sec:uppBnd-SI-1}
		\E\,\left( \int^{T}_\zs{0}\, f_\zs{t}\d W_\zs{t}\right)^{2m}\le\,(m(2m-1))^{m}\,T^{m-1}\,
		\int^{T}_\zs{0}\,  \E\,f^{2m}_\zs{t}\d t\,.
	\end{equation}
\end{lemma} 

\noindent This lemma is shown in \cite[Lemma 4.12]{LipShir1}.

\subsection{Proof of Lemma \ref{Le.sec:App-33}}

\noindent First of all we need the following result shown in
\cite[Lemma A7]{BenAlayaetal2025}.
\begin{lemma}\label{Le.sec:App-2}
	For any $r>2$ and any compact set $\Theta\subset (\sigma/2,+\infty)\, \times\, (0,+\infty)$ for the matrices 
	\eqref{sec:MultPr-1}		
	and \eqref{sec:Mlt.cond-1} the following property holds true
	\begin{equation}\label{sec:uppBnd-Dev-1}
		\d_\zs{*}=
		\sup_\zs{z\ge 1}
		\sup_\zs{\theta\in\Theta}\,
		\E_\zs{\theta}\,
		\left(
		\sqrt{z}\,
		\left|\frac{ G_\zs{\t_\zs{z}}}{z}-\frac{F}{\tr F}\right|\right)^{r}<\infty\,.
	\end{equation}
\end{lemma}

\noindent
First of all, note that from the definitions \eqref{sec:MultPr-b-bnn-101} and \eqref{sec:MultPr-AsPrs-bb}  we can deduce directly that $\b_\zs{n}\le 1$ and $\b_\zs{*}\le 1$. Therefore, 
$$
\left| \b^{2}_\zs{n}-\b^{2}_\zs{*}\right|
\le 
2 
\left| \b_\zs{n}-\b_\zs{*}\right|
\le 
2 \left| \b_\zs{n}-\b_\zs{*}\right|\,\Chi_\zs{\{\lambda_\zs{min}(G_\zs{\t_\zs{n}})>0\}}
+
2\,\Chi_\zs{\{\lambda_\zs{min}(G_\zs{\t_\zs{n}})=0\}}
\,.
$$
\noindent Note here that on the set $\{\lambda_\zs{min}(G_\zs{\t_\zs{n}})>0\}$ the first difference can be estimated as
$$
\left| \b_\zs{n}-\b_\zs{*}\right|\,
\le \left| 
\D_\zs{n}
\right|
\quad\mbox{and}\quad
\D_\zs{n}=\frac{ G_\zs{\t_\zs{n}}}{\kappa_\zs{n}}-\frac{F}{\tr F}
\,.
$$
\noindent Moreover, note that for any $\theta\in\Theta$
$$
\P_\zs{\theta}\left(\lambda_\zs{min}(G_\zs{\t_\zs{n}})=0 \right)=
\P_\zs{\theta}\left(\lambda_\zs{min}\left(\frac{G_\zs{\t_\zs{n}}}{\kappa_\zs{n}}\right)=0 \right)
\le 
\P_\zs{\theta}\left(
\left| \D_\zs{n} \right|\ge 
\lambda_\zs{*} 
\right)\,,
$$
\noindent where $\lambda_\zs{*}=\min_\zs{\theta\in\Theta}\lambda_\zs{min}(F)/\tr F>0$.  
Using here the  Chebyshev inequality and the bound \eqref{sec:uppBnd-Dev-1} we can deduce  that for any $r>2$
$$
\sup_\zs{\theta\in\Theta}\,
\P_\zs{\theta}\left(\lambda_\zs{min}(G_\zs{\t_\zs{n}})=0 \right)
\le 
\frac{
	\sup_\zs{\theta\in\Theta}\,
	\E_\zs{\theta} \left| \D_\zs{n} \right|^{r}}{\lambda_\zs{*}^{r}}
\le \frac{\d_\zs{*}}{\lambda_\zs{*}^{r}} \kappa^{-r/2}_\zs{n}\,.
$$
\noindent 
Therefore,  for any $r>2$ the random variables $\psi_\zs{n}=\sqrt{\kappa_\zs{n}}(\b^{2}_\zs{n}-\b^{2}_\zs{*})$ 
can be estimated from above as
$$
\psi_\zs{*}=
\sup_\zs{n\ge 1}\,
\sup_\zs{\theta\in\Theta}\,
\E_\zs{\theta} |\psi_\zs{n}|^{r}
<\infty
$$

\noindent 
{and in view of
	the definition \eqref{sec:MultPr-b-nn-122} one can deduce} that for any $n>\u_\zs{*} H$ and $r>2$
\begin{align*}
	\P_\zs{\theta}\,
	\left(
	\upsilon_\zs{H}
	>n
	\right)&=
	\P_\zs{\theta}\,
	\left(
	\sum^{n}_\zs{k=1}\,\b^{2}_\zs{k}
	<H
	\right)
	\le 
	\P_\zs{\theta}\,
	\left(
	\sum^{n}_\zs{k=1}\,\frac{|\psi_\zs{k}|}{\sqrt{\kappa_\zs{k}}}
	> \b^{2}_\zs{*}n- H
	\right)
	\\[2mm]&
	\le 
	\frac{1}{(\b^{2}_\zs{*}n- H)^{r}}
	\,\E_\zs{\theta}\left(\sum^{n}_\zs{k=1}\,\frac{|\psi_\zs{k}|}{\sqrt{\kappa_\zs{k}}}\right)^{r}
	\,.
\end{align*}
\noindent  
{Through the  H\"older inequality the sum in the last expectation can be estimated as }
$$
\left(\sum^{n}_\zs{k=1}\,\frac{|\psi_\zs{k}|}{\sqrt{\kappa_\zs{k}}}\right)^{r}
\le \left(\sum^{n}_\zs{k=1}\,\frac{1}{\sqrt{\kappa_\zs{k}}}\right)^{r-1}
\left(\sum^{n}_\zs{k=1}\,\frac{|\psi_\zs{k}|^{r}}{\sqrt{\kappa_\zs{k}}}\right)\,.
$$
\noindent Therefore, for $n>\u_\zs{*} H$
$$
\sup_\zs{\theta\in\Theta}\,
\P_\zs{\theta}\,
\left(
\upsilon_\zs{H}
>n
\right)
\le 
\frac{\psi_\zs{*}}{(n- \u_\zs{*} H)^{r}}\,
\left(\sum^{n}_\zs{k=1}\,\frac{1}{\sqrt{\kappa_\zs{k}}}\right)^{r}\,.
$$
\noindent Using here the conditions \eqref{sec:seq-kappa}-\eqref {sec:seq-kappa--0}, we obtain
that		
$$
\sum^{n}_\zs{k=1}\,\frac{1}{\sqrt{\kappa_\zs{k}}}
\le \frac{\n^{*}_\zs{H}}{\sqrt{H}}
+
\sum^{n}_\zs{k=1}\,\frac{1}{\sqrt{\kappa^{*}_\zs{k}}}
\le 2\u_\zs{*} \sqrt{H}
+n^{\delta^{*}}
\sup_\zs{n\ge 1}\,n^{-\delta^{*}}
\sum^{n}_\zs{k=1}\,\frac{1}{\sqrt{\kappa^{*}_\zs{k}}}\,.
$$		
\noindent This implies
the upper bound  \eqref{sec:uppBnd-tail--0}. 
\endproof


\begin{thebibliography}{}
	\bibitem[Ben Alaya and Kebaier(2013)]{BAK2013} 
	Ben Alaya, M. and Kebaier, A. 2013 ``Asymptotic Behavior of the Maximum Likelihood Estimator for Ergodic and Nonergodic Square-Root Diffusions``, \emph{Stochastic Analysis and Applications}, Vol. {\bf 31}, pp. 552-573.
	
	\bibitem[Ben Alaya and Kebaier(2010)]{BAK2010} 
	Ben Alaya, M. and Kebaier, A. 2010 ``Parameter Estimation for the Square-Root Diffusions: Ergodic and Nonergodic Cases``, \emph{Stochastic Models}, Vol. {\bf 28}.
	
	
	\bibitem[Ben Alaya, Ngô and Pergamenchtchikov(2025)]{BenAlayaetal2025}
	Ben Alaya, M., Ng\^o, T-B-T. and Pergamenchtchikov, S. 2025 ``Optimal guaranteed estimation methods for the Cox - Ingersoll - Ross models``, 
	\emph{Stochastics: An International Journal Of Probability And Stochastic Processes}, be published. \url{https://doi.org/10.1080/17442508.2025.2450219}.
	
	
	\bibitem[Berdjane and Pergamenshchikov(2013)]{BerdjanePergamenshchikov2013}
	Berdjane, B. and Pergamenshchikov, S. M. 2013
	``Optimal consumption and investment
	for  markets with  randoms coefficients``.
	-\emph{ Finance and Stochastics},  {\bf 17}, 2, 419--446.
	
	
	
	\bibitem[Berdjane and Pergamenshchikov(2015)]{BerdjanePergamenshchikov2015}
	Berdjane, B. and Pergamenshchikov, S. M. 2015
	``Sequential  $\delta$ - optimal consumption and investment
	for stochastic volatility markets with unknown parameters``.
	- \emph{ Theory of Probability and its Applications},   
	{\bf 60} (4), p. 628 -- 659.
	
	
	
	\bibitem[Cox, Ingersoll and Ross(1985)]{CoxIngersollRoss1985}
	Cox, J.C., Ingersoll, J.E., Ross, S.A. 1985 ``A theory of the term structure of interest rates``. 
	\emph{Econometrica}, {\bf 53} (2), 385–407.
	
	
	
	
	\bibitem[Galtchouk and Pergamenshchikov(2007)]{GaltchoukPergamenshchikov2007} 
	Galtchouk, L.I. and Pergamenshchikov, S.M.
	Uniform concentration inequality for ergodic diffusion processes. -
	{\sl Stochastic processes and their applications}, 2007, v. {\bf 7},
	p. 830-839.
	
	\bibitem[Galtchouk and Pergamenshchikov(2011)]{GaltchoukPergamenshchikov2011} 
	Galtchouk, L.I. and Pergamenshchikov, S.M.
	Adaptive sequential estimation for
	ergodic diffusion processes in quadratic metric.
	{\em Journal of Nonparametric Statistics}, 2011, 
	{\bf 23}, 2, 255-285.
	
	\bibitem[Galtchouk and Pergamenshchikov(2015)]{GaltchoukPergamenshchikov2015} 
	Galtchouk, L.I. and Pergamenshchikov, S.M.
	Efficient pointwise estimation based on discrete data
	in ergodic nonparametric diffusions.  {\sl Bernoulli},  2015,  {\bf 21} (4),  2569 - 2594.
	
	\bibitem[Galtchouk and Pergamenshchikov(2022)]{GaltchoukPergamenshchikov2022} 
	Galtchouk, L.I. and Pergamenshchikov, S.M.
	Adaptive
	efficient analysis for big data ergodic diffusion models. - 
	{\em Statistical inference for stochastic processes}, 2022, {\bf  25}, 127 -- 158
	
	
	
	\bibitem[Heston(1993)]{Heston1993} 
	Heston, S.L. 1993
	``A closed-form solution for options with stochastic volatility, with applications to bond and currency options``. \emph{Review of Financial studies} , {\bf 6} (2), 327–343.
	
	\bibitem[Ibragimov and Has'minskii(1981)]{IbraHas81} Ibragimov, I.A. and  Has'minskii, R.Z. 1981 ``Statistical Estimation: Asymptotic Theory``, \emph{Springer New York}.
	
	
	
	
	
	
	
	
	
	
	\bibitem[Konev and Pergamenshchikov(1990)]{KonevPergamenshchikov1990}
	Konev, V.V. and Pergamenshchikov, S. M. 1990 ``Truncated Sequential
	Estimation of the Parameters in a Random Regression``.- \emph{Sequential
		Analysis}, {\bf 9}, 1, 19 - 41.
	
	
	\bibitem[Konev and Pergamenshchikov(1992)]{KonevPergamenshchikov1992}
	Konev, V.V. and Pergamenshchikov, S. M. 1992
	``On Truncated Sequential
	Estimation of the Parameters of Diffusion Processes``.  -  \emph{ Methods  of
		Economical Analysis},  Central   Economical  and
	Mathematical Institute of Russian Academy of Science,
	Moscow,  p. 3-31.
	
	
	
	
	
	
	\bibitem[Lamberton and Lapeyre(1997)]{LamLap97}
	Lamberton, D. and Lapeyre, B. 1997 ``Introduction au Calcul Stochastique Appliqu\'e à la Finance``, \emph{Ellipses \'Edition Marketing: Paris}, 2nd edition.
	
	
	\bibitem[Liptser and Shiryaev(2001)]{LipShir1} 
	Liptser, R. S. and {Shiryaev}, Albert N. 2001 ``Statistics of Random Processes I ``, \emph{Springer-Verlag Berlin Heidelberg,} Vol. 5, Stochastic Modelling and Applied Probability, 2nd edition.
	
	
	
	\bibitem[Nguyen and Pergamenshchikov(2017)]{NguyenPergamenshchikov2017}
	Nguyen, T. H. and Pergamenshchikov, S. M. 2017 ``Approximate hedging problem with transaction costs in stochastic volatility markets``. -
	\emph{ Mathematical Finance}, Vol. 27, No. 3, 832–865.
	
	
	
	
	\bibitem[Novikov and Shiryaev and Kordzakhia(2024)]{NovikovShiryaevKordzakhia2024}
	Novikov, A.A.,  Shiryaev, A.N.  and Kordzakhia, N.E. 2024  "On parameter estimation of diffusion type processes: sequential estimation revisited".
	\emph{Theory Probab. Appl.,  Society for Industrial and Applied Mathematics} {\bf  69}, No. 3 
	
	
	
	\bibitem[Pergamenchtchikov, Tartakovsky and Spivak(2022)]{PergamenchtchikovTartakovskySpivak2022}
	Pergamenchtchikov, S. M., Tartakovsky, A. and Spivak, V. 2022
	``Minimax and pointwise sequential changepoint detection and identification for general stochastic models``. 
	-  \emph{ Journal of Multivariate Analysis},  {\bf 190}, 104977.
	
	
	
	
	
	
\end{thebibliography}
\end{document}